\documentclass[10pt,reqno,a4paper]{amsart}%
\usepackage[a4paper,margin=3cm]{geometry}%
\usepackage[pagebackref]{hyperref}%
\usepackage{mathtools,amssymb,tikz,lmodern}%
\usetikzlibrary{decorations.pathreplacing}
\usepackage[shortlabels]{enumitem}%
\usepackage[utf8]{inputenc}%
\usepackage[T1]{fontenc}%
\newcommand{\ee}{\mathrm{e}}
\newcommand{\ii}{\mathrm{i}}
\newcommand{\dd}{\mathrm{d}}
\newcommand{\ind}{\mathbf{1}}
\newcommand{\rL}{\mathrm{L}}

\newcommand{\dP}{\mathbb{P}}
\newcommand{\dR}{\mathbb{R}}

\newcommand{\cE}{\mathcal{E}}

\newcommand{\cW}{\mathcal{W}}
\newcommand{\cZ}{\mathcal{Z}}

\title[At the edge of a one-dimensional jellium]%
{At the edge of a one-dimensional jellium}

\author{Djalil Chafaï}%
\address[DC]{CEREMADE, Université Paris-Dauphine, PSL University, France.}
\email{\url{mailto:djalil(at)chafai.net}} %
\urladdr{\url{http://djalil.chafai.net/}}

\author{David García-Zelada}%
\address[DGZ]{Aix-Marseille University, Institut de Mathématiques de Marseille
  (I2M), France}
\email{\url{mailto:david.garcia-zelada{at}univ-amu.fr}}
\urladdr{\url{https://davidgarciaz.wixsite.com/math}}

\author{Paul Jung}
\address[PJ]{KAIST, Daejeon, South Korea.}
\email{\url{mailto:pauljung(at)kaist.ac.kr}}
\urladdr{\url{http://mathsci.kaist.ac.kr/~pauljung/}}

\date{Autumn 2020, revised Spring 2021, compiled \today.}

\newtheorem{theorem}{Theorem}[section]%
\newtheorem{corollary}[theorem]{Corollary}%
\newtheorem{proposition}[theorem]{Proposition}%
\newtheorem{lemma}[theorem]{Lemma}%
\theoremstyle{definition}
\newtheorem{example}[theorem]{Example}%
\newtheorem{remark}[theorem]{Remark}%

\numberwithin{equation}{section}
\numberwithin{figure}{section}



\begin{document}

\begin{abstract}
  We consider a one-dimensional classical Wigner jellium, not necessarily
  charge neutral, for which the electrons are allowed to exist beyond the
  support of the background charge. The model can be seen as a one-dimensional
  Coulomb gas in which the external field is generated by a smeared background
  on an interval. It is a true one-dimensional Coulomb gas and not a
  one-dimensional log-gas. The system exists if and only if the total
  background charge is greater than the number of electrons minus one. For
  various backgrounds, we show convergence to point processes, at the edge of
  the support of the background. In particular, this provides asymptotic
  analysis of the fluctuations of the right-most particle. Our analysis
  reveals that these fluctuations are not universal, in the sense that
  depending on the background, the tails range anywhere from exponential to
  Gaussian-like behavior, including for instance Tracy\,--\,Widom-like
  behavior. We also obtain a Rényi-type probabilistic representation for the
  order statistics of the particle system beyond the support of the
  background.
\end{abstract}

\maketitle

{\footnotesize\tableofcontents}

\section{Introduction}

Introduced by Wigner in \cite{wigner1934interaction, TF9383400678} for
modeling electrons in metals, the jellium is a Coulomb gas of like-signed
equally charged particles for which an external potential is induced by a
background of smeared charge with opposite sign. The model was inspired by the
Hartree\,--\,Fock model of quantum mechanics. This model and its variants go
by many other names, including the \emph{one-component plasma} or
\emph{uniform electron gas}, see for instance \cite{MR3732693}. Typically, one
imposes the constraint that all charged particles live in some compact region
which is equivalent to imposing an infinite external potential on the
complement of this compact region. Charge neutrality is also usually assumed,
in other words the total charge of the background matches the number of
particles. These restrictions ensure that the system exists and the
mathematical interest typically focuses on the limiting system as the volume
of the compact set (the background) goes to infinity (thermodynamic limit).
The classical one-dimensional jellium has been rigorously studied by Baxter
\cite{baxter_1963} who found the partition function exactly, by Kunz
\cite{KUNZ1974303} who showed the Wigner lattice (crystallization) exists for
all temperatures, by Aizenman and Martin \cite{aizenman1980structure}, and by
Aizenman, Goldstein, and Lebowitz \cite{aizenman2001bounded}, among others. In
the quantum case, Brascamp and Lieb \cite{Brascamp2002} represented the
partition function exactly and showed crystallization when the inverse
temperature parameter $\beta$ is large enough, while the proof of
crystallization for all temperatures is obtained in \cite{jansen2014wigner}
(see also \cite{PhysRevB.94.115417} for thermal effects on crystallization).

In this work, we consider the classical jellium but we do not assume that the
particles are restricted to live on the region where there is background
charge, nor do we assume that the whole system is charge neutral. The system
is well-defined if and only if the total background charge is greater than the
number of particles minus one. We will assume that this condition is
satisfied. The limiting behavior of the particles in the bulk does not change
by allowing the particles to leave the background region, thus we focus our
attention on the edge of the system, near where the background charge ends.
One may then view the system as a jellium on a half-space. More importantly,
similar extremal analysis has been carried out for many similar models such as
the one-dimensional log-gas and two-dimensional unconfined jellium. Thus it is
natural, in the above described setting, to analyze the asymptotic location of
the particles farthest away from the origin.

The backgrounds we consider below have a finite total charge which is allowed
to grow as $n \to \infty$. Edge statistics for a related model one-dimensional
jellium, with infinite background charge, were analyzed in
\cite{PhysRevLett.119.060601, MR3817495}. In the case of a uniform background
with support growing ``fast enough'' we obtain a system behaving similarly to
theirs; however, we will see that in general one may obtain a range of varying
behaviors at the edge.

Let us now describe our model. The Coulomb kernel in dimension $d=1$ is
\begin{equation}\label{eq:couker}
	g(x) =- \frac{|x|}{2},\quad x\in\dR,
\end{equation}
which is the fundamental solution of the Poisson equation $\Delta g=-\delta_0$
in the sense of distributions. More precisely for all smooth and compactly
supported $\varphi: \mathbb R \to \mathbb R$,
\[
  \int_{\mathbb R}
  g(x) \frac{\dd^2}{\dd x^2}\varphi(x) \dd x
  = -\varphi(0).
\]
Let $\mu=\mu_+-\mu_-$ be a possibly signed measure on $\dR$ with finite first absolute
moment, namely $g\in\rL^1(|\mu|)$ where $|\mu|=\mu_++\mu_-$. 
The Coulomb potential generated at the point $x\in\dR$ by $\mu$ is
\begin{equation}\label{eq:U}
	U_\mu(x)=(g*\mu)(x)= -\int\frac{|x-y|}{2}\mu(\dd y),
\end{equation}
which satisfies $\Delta U_\mu= -\mu$ in the sense that for all smooth and
compactly supported $\varphi: \mathbb R \to \mathbb R$,
\begin{equation}\label{eq:inv}
  \int_{\mathbb R} U_\mu(x) \frac{\dd^2}{\dd x^2}\varphi(x) \dd x
  = -\int \varphi \, \dd \mu,
\end{equation}
see for instance \cite{MR3308615,MR0350027}. The Coulomb (self-interaction)
energy of $\mu$ is defined by
\begin{equation}\label{eq:E}
  \cW(\mu)
	=\frac{1}{2}\iint g(x-y)\mu(\dd x)\mu(\dd y)
	=\frac{1}{2}\int U_\mu(x)\mu(\dd x).
\end{equation}
The electric field generated at the point $x\in\mathbb{R}$ by a (possibly
signed) measure $\mu$ on $\mathbb{R}$ is
\begin{equation}\label{eq:electric}
  \cE_\mu(x)=-\frac{\dd}{\dd x}U_\mu(x)
  =\frac{1}{2}\int\mathrm{sign}(x-y)\dd\mu(y).
\end{equation}
For all $n\geq1$, we consider $n$ unit negatively charged particles
(electrons) at positions $x_1,\ldots,x_n$ in $\dR$, lying in a positive
background of total charge $\alpha>0$. The background is smeared according to
a probability measure $\rho$ on $\mathbb{R}$ with finite Coulomb energy
$\cW(\rho)$. We could alternatively suppose that the particles are positively
charged (cations) and the background is negatively charged; this
reversed choice would not affect the analysis of the model. The total energy
of the system is
\[
  -\frac12 \sum_{i<j} |x_i-x_j|%
  -\alpha\sum_{i=1}^nU_{\rho}(x_i)%
    +\alpha^2\cW(\rho).
\]
However, the term $\alpha^2\cW(\rho)$ will not be that important for our
analysis so we set
\begin{equation}\label{eq:Hn}
  H_n(x_1,\ldots,x_n)
  =-\frac12 \sum_{i<j} |x_i-x_j|%
  -\alpha\sum_{i=1}^nU_{\rho}(x_i).
\end{equation}
The one-dimensional Coulomb model comes with remarkable identities such as, for all $x\in\mathbb{R}$,
\begin{equation}\label{eq:PotentialExplicitFormula}
 \int \frac{|x-y|}{2} \dd \rho(y)				
  =  \frac{x}{2} +\int(y-x)\mathbf{1}_{(x,\infty)}\,\dd \rho(y)
	- \int \frac{y}{2}\,\dd \rho(y).
\end{equation}
In the same spirit, \emph{Baxter's combinatorial identity} \cite{baxter_1963}
states that for all $(x_1,\ldots,x_n)\in\dR^n$,
\begin{equation}\label{eq:Baxter}
  -\sum_{i < j} |x_i - x_j|
  =\sum_{i<j}(x_{(j)}-x_{(i)})
  =\sum_{k=1}^n (2k-n-1) x_{(k)},
\end{equation}
where $x_{(n)}\leq\cdots\leq x_{(1)}$ is the reordering of $x_1,\ldots,x_n$;
in particular,
\begin{equation}\label{eq:minmax}
  x_{(n)}=\min_{1\leq i\leq n}x_i
  \quad\text{and}\quad
  x_{(1)}=\max_{1\leq i\leq n}x_i,
\end{equation}
which allows to rewrite the energy as
\begin{equation}\label{eq:CombinatorialIdentity}
  H_n(x_1,\dots,x_n)
  =
  \sum_{k=1}^n 	
  \Bigr[\frac{2k-n-1}{2}x_{(k)}-\alpha_n U_\rho(x_{(k)})\Bigr].
\end{equation}

For simplicity, we assume in the whole text that $\rho$ is absolutely
continuous with respect to Lebesgue measure, with a density function still denoted
$\rho$ by a slight abuse of notation. We say that the system is \emph{neutral}
in charge when $\alpha=n$, and that the background is \emph{uniform} when
$\rho$ is the uniform distribution on an interval $[a,b]$. For all
$\beta>0$, we set
\begin{equation}\label{eq:Z}
  \cZ_n=\int_{\dR^n}\ee^{-\beta H_n(x_1,\ldots,x_n)}
  \dd x_1\cdots\dd x_n\in[0,\infty].
\end{equation}
It can be checked that $\cZ_n<\infty$ if and only if $\alpha>n-1$, see
\cite[Lemma 2.1]{jellium1d-arxiv-v1}.

When $\alpha>n-1$, we can then define the Boltzmann\,--\,Gibbs probability measure
$P_n$ on $\dR^n$ by
\begin{equation}\label{eq:P}
  \dd P_n(x_1,\ldots,x_n)
  =\frac{\ee^{-\beta H_n(x_1,\ldots,x_n)}}{\cZ_n}
  \dd x_1\cdots\dd x_n.
\end{equation}
This is called a \emph{Coulomb gas} with external potential
$V=-\frac{\alpha}{n}U_{\rho}$. We are dealing with electrostatics in the sense
that the charges do not move. In our setting, this external potential arises
from a smeared background with distribution $\rho$, however, one can define
general Coulomb gases for any confining external potential $V$ for which
$\frac{n}{\alpha} \Delta V$ may not necessarily be a probability measure. Let
\begin{equation}\label{eq:Xn}
   \mathbf{X}^{(n)}=(X_1^{(n)},\ldots,X_n^{(n)})\sim P_n.
\end{equation}

\begin{example}[Coulomb gas with quadratic external field]\label{ex:quad}
  Let us consider the example for which $\rho$ is the uniform probability
  measure on an interval $[a,b]$ with $a<b$. Then, for all $x\in\dR$,
  \begin{equation}\label{eq:Unif}
    -U_\rho(x)
    =\frac{1}{2(b-a)}\int_a^b|x-y|\dd y
    =\begin{cases}
      \displaystyle\frac{\left|x-\frac{a+b}{2}\right|}{2} &\text{if $x\not\in[a,b]$}\\[1em]
      \displaystyle\frac{\left(x-\frac{a+b}{2}\right)^2+\frac{(b-a)^2}{4}}{2(b-a)}&\text{if $x\in[a,b]$}
    \end{cases}.
  \end{equation}
  The potential $V=-\frac{\alpha}{n}U_\rho$ then behaves quadratically on
  $[a,b]$ and is affine outside $[a,b]$. Conditioned on all the particles
  lying inside $[a,b]$, it is possible to interpret $P_n$ as a conditioned
  Gaussian law. Using Baxter's identity \eqref{eq:Baxter} together with
  \eqref{eq:Unif}, we have that when $\{x_1,\ldots,x_n\}\subset[a,b]$,
  \begin{equation}\label{eq:BaxterH}
    H_n(x_1,\ldots,x_n)
	  = \sum_{k=1}^n\frac{2k-n-1}{2}x_{(k)}
	    +\frac{\alpha}{2(b-a)}\sum_{i=1}^n\Bigr(x_{(i)}-\frac{a+b}{2}\Bigr)^2+\frac{n\alpha(b-a)}{8}.
  \end{equation}
  This formula shows then that $\mathbf{X}^{(n)}\sim P_n$ is conditionally Gaussian
  in the sense that
  \begin{multline}\label{eq:PnG}
    \mathrm{Law}\Bigr((X_{(n)},\ldots,X_{(1)})\bigm\vert \{X_1,\ldots,X_n\}\subset[a,b]\Bigr)\\
    =\mathrm{Law}\Bigr((Y_n,\ldots,Y_1)\bigm\vert a\leq Y_n\leq\cdots\leq Y_1\leq b\Bigr)
  \end{multline}
  where $Y_1,\ldots,Y_n$ are independent real Gaussian random variables with
  \[
    \mathbb{E}Y_k=\frac{a+b}{2}+\frac{b-a}{2\alpha}\left(n+1-2k \right) 
    \quad\text{and}\quad
    \mathbb{E}((Y_k-\mathbb{E}Y_k)^2)=\frac{b-a}{\alpha \beta}.
  \]
  This was already noted in \cite{baxter_1963}. Now if we consider the limit
  $a\to-\infty$, $b\to \infty$ with $\alpha/(b-a) \to c > 0$, then $P_n$ can
  be interpreted as a Coulomb gas for which the potential is quadratic
  everywhere, namely $V=\frac{c}{2n}\left|\cdot\right|^2$. Since the second
  derivative of $V$ is a constant, this can also be seen as a jellium with a
  background equal to a multiple of Lebesgue measure on the whole of $\dR$.
  Note that this jellium is not neutral, but rather, has an infinite charge
  imbalance for every $n$. Under the scaling $x_i=\sqrt{n}y_i$, see Remark
  \ref{rm:scaling}, this limiting case matches the model studied by
  \cite[Equation (14)]{MR3817495} but note that their $\alpha$ plays the role
  of our $c$ up to a dilation. This Coulomb gas model with quadratic external
  field in one dimension is analogous to the complex Ginibre ensemble which is
  a Coulomb gas in two dimensions.
\end{example}

\begin{remark}[Scale invariance]\label{rm:scaling}
  The model \eqref{eq:P} has a scale invariance which comes from the
  homogeneity of the one-dimensional Coulomb kernel. More precisely, if we
  denote by $\mathrm{dil}_\sigma(\mu)$ the law of the random vector
  $\sigma \mathbf{X}$ when $\mathbf{X}\sim\mu$, then, for all $\sigma>0$, dropping the $n$
  subscript on $P$,
  \[
    \mathrm{dil}_\sigma(P^{\alpha,\beta,\rho})
    =P^{\alpha,\frac{\beta}{\sigma},\mathrm{dil}_\sigma(\rho)}.
  \]
  In other words, if $\mathbf{X}^{(n)}\sim P^{\alpha,\beta,\rho}$ then
  $\sigma \mathbf{X}^{(n)}\sim
  P^{\alpha,\frac{\beta}{\sigma},\mathrm{dil}_\sigma(\rho)}$. This is
  useful in the asymptotic analysis of the model as $n\to\infty$, and reveals
  the special role played by $\alpha$ as a shape parameter. Here the inverse temperature $\beta$ is a
  scale parameter, in contrast with the situation for log-gases.
\end{remark}

\subsection*{Structure of the paper}

\begin{itemize}
  \item Section \ref{sec:results} states our main results: Theorems
    \ref{th:edge pp neutral}, \ref{th:edge:nn:pp}, \ref{th:edge:neutral:pp}
    and Corollary \ref{cor:tail}.
  \item Section \ref{sec:proofs} proves our main results. This is done in two
    steps. In Section \ref{sec:icg}, we first discuss the right-hand sides of
    equations \eqref{eq:unconditioned limit} and \eqref{eq:limit weakly
      confining} contained in our main results. In Section \ref{sec:proofs},
    we complete the proofs by showing that the left-hand sides of these two
    equations converge to the right-hand sides, and we also prove the behavior
    of the single right-most particle as described in Corollary
    \ref{cor:tail}.
  \item Appendices \ref{se:tail} and \ref{se:stodom} give results about tail
    asymptotics and stochastic domination used in the proofs of our main
    results.
  \item Our main results concern point processes with infinitely many
    particles. If one only needs conditional results about the right-most
    particle (or finitely many particles), the proofs can be greatly
    simplified. Appendix \ref{sec:conditional} illustrates this
    simplification.
\end{itemize}

\section{Main results}\label{sec:results}

As a prerequisite to studying the edge asymptotics, one should first verify
that, at the macroscopic level, the limiting equilibrium measure is equal to
$\rho$ as one would naturally expect. For those interested in precise details
of such global asymptotics, we refer to \cite[Theorem
2.2]{jellium1d-arxiv-v1}.

Our main results show that in the limit, as $n\to\infty$, one can obtain an
infinite point process at the edge of the jellium. These point processes can
be interpreted as infinite-volume Gibbs measures when the background region is
expanded in only one direction (leaving the other end of the background
fixed). It will be apparent in our proofs, that the limits are indeed
Gibbsian, in the sense that the limit does not depend on how one takes these
infinite limits, as long as one end is fixed. Our first Theorem \ref{th:edge
  pp neutral} does this in the natural setting of a growing uniform background
$\rho_n$, with an asymptotically neutral system similar to the original model
of \cite{baxter_1963}. Since our results concern the edge of the jellium, it
is natural to require $\rho_n$ to be supported on $(-\infty,0]$. Under the
conditions of the next theorem,
$\frac{1}{n}\sum_{i=1}^n \delta_{n^{-1}X_{i}^{(n)}}$ converges to the uniform
measure on $[-1,0]$.

\begin{theorem}[Point process at the edge, asymptotically neutral regime]\label{th:edge pp neutral}%
  Suppose
  \begin{itemize}
  \item $\beta>0$ is fixed;
  \item $\alpha_n - (n-1) = 2\lambda 
    \in (0,\infty)$;
  \item $\rho_n$ is uniform on $[-\alpha_n,0]$.
  \end{itemize}
  If $\mathbf{X}^{(n)} \sim P_n$ as in 
  \eqref{eq:Xn}, then
  \begin{align}\label{eq:limit weakly confining}
    \lim_{n\to\infty}\mathrm{Law}\big(X^{(n)}_{(k)},\dots,X_{(1)}^{(n)}\big)
    = \lim_{n\to\infty}\mathrm{Law}\big(Y_k,\ldots, Y_1\mid Y_n<\cdots<Y_1\big)
  \end{align}
  where $\{Y_i\}_{i\ge 1}$ are independent random variables such that $Y_i$
  has a density proportional to
  \[
    \exp\Bigr(-\beta \Bigr[ (i -1 + \lambda)x
        +\frac{ x^2}{2}\ind_{(-\infty,0)}\Bigr]\Bigr).
  \]
\end{theorem}

The result should also hold when
$\lim_{n\to\infty}(\alpha_n-(n-1))=2\lambda\in(0,\infty)$, however, for
simplicity, we have only considered the special case
$\alpha_n - (n-1) = 2\lambda$. It seems that the more general case amounts to
proving the right continuity, with respect to $\lambda$, of the limiting
process.

\begin{remark}[Relaxing uniformity of background]\label{rk:uniform}
  Since we are only interested in the edge behavior, one should be able to
  relax the assumption that $\rho_n$ is uniform at the left of $0$ by
  requiring that $\lim_{n \to \infty}\alpha_n\rho_n$
  in the vague sense is the Lebesgue measure restricted to $(-\infty,0]$.
\end{remark}

\begin{remark}[Crystallization]
	The condition on $\alpha_n$ in Theorem \ref{th:edge pp neutral} requires
	$\alpha_{n+1}-\alpha_n=1+o(1)$. Indeed, due to crystallization and
	translation symmetry breaking of the one-dimensional jellium
	\cite{aizenman1980structure, aizenman2010symmetry, jansen2014wigner}, one is
	not allowed to continuously increase the background charge when taking the
	thermodynamic limit, but rather, the increases must be in (roughly) integer
	steps.
\end{remark}

We next consider a (generalized) model similar to that studied in
\cite{PhysRevLett.119.060601, MR3817495,schehr-et-al-1d} for the quadratic
Coulomb gas model of Example \ref{ex:quad}. In our case, we will use a
background with finite total charge, but we allow the total background charge
$\alpha_n$ which grows at a rate faster than $n$. This has a similar affect to
first growing the background to obtain a quadratic external potential, and
then taking the number of particles to infinity. We will see that the special
case where $\rho_n$ is uniform (with $\alpha_n$ growing faster than $n$)
shares the same features as the Coulomb gas with external quadratic potential
\cite[Equation (14)]{MR3817495}; however this sort of behavior is not seen in
general. In order to get such Gaussian behavior it is necessary for the
background charge to extend uniformly beyond the region where the extremal
particles live. This is the special case $\gamma=2$, in our next result. For
general nonneutral systems, one may interpolate between exponential and
Gaussian tails, and even beyond, by varying the decay of the background at the
edge of its support. Under the conditions of the next theorem,
$\frac{1}{n}\sum_{i=1}^n \delta_{n^{-1}X_{i}^{(n)}}$ converges to the uniform
measure on $[-1,0]$ as in Theorem \ref{th:edge pp neutral}.

\begin{theorem}[Point process at the edge, nonneutral regime]
  \label{th:edge:nn:pp}%
  Suppose that 
  \begin{itemize}
  \item $\beta>0$ is fixed;
  \item $\alpha_n$ is such that $\lim_{n\to\infty}(\alpha_n-(n-1))=\infty$;
  \item $\rho_n$ is such that for some fixed $\gamma>1$,
    \[
      \alpha_n\rho_n(x)
      =\ind_{[-\frac{\alpha_n+n}2,0]}(x)\dd x
      +(\gamma-1)x^{\gamma-2}\ind_{\left[0,(\frac{\alpha_n-n}2)^{1/(\gamma-1)}\right]}(x)\dd x.
    \]
  \end{itemize}
  If $\mathbf{X}^{(n)} \sim P_n$ as in 
  \eqref{eq:Xn}, then \eqref{eq:limit weakly confining} holds except
  now $\{Y_i\}_{i\ge 1}$ are independent random variables such that $Y_i$ has
  a density proportional to
  \[
    \exp\Bigr(-\beta\Bigr[\left(i-\frac{1}{2}\right)x+
        \frac{x^2}{2}
        \ind_{(-\infty,0)} +
        \frac{x^{\gamma}}{\gamma}
        \ind_{(0,\infty)}\Bigr]\Bigr)\mathrm .
  \]
\end{theorem}

When $\gamma\in(1,2)$ there is an (integrable) singularity in the background
density $\rho_n$ at $0$. This singularity is not important to the edge
behavior, but rather serves only to give the density of $Y_i$ a clean form.
One could smooth out this density at the cost of complicating the density of
$Y_i$.

\begin{remark}[Gaussian and Tracy-Widom like cases]\label{rk:bulk}
	When $\gamma=2$, then Theorem \ref{th:edge:nn:pp} is about the model of
	Example \ref{ex:quad} and represents the particle locations as Gaussian
	variables conditioned on living on a convex simplex. On the other hand, if
	$\gamma=3/2$ in Corollary \ref{cor:tail}, one obtains the tail behavior of
	the Tracy\,--\,Widom distribution $\mathrm{TW}_\beta$ where we recall from
	\cite{MR2813333,dumaz2013right, MR2967965},
	\[
	\dP(M>t)=\ee^{-\frac23 \beta t^{3/2}(1+o_{t\to\infty}(1))}
	\quad\text{for}\quad M\sim\mathrm{TW}_\beta.
	\]
\end{remark}

The final situation we consider is to let $\rho$ be fixed. Since the total
background is $\alpha_n\rho$, this amounts to growing the background
vertically on the support of $\rho$. Under the conditions of the next theorem,
$\frac{1}{n}\sum_{i=1}^n \delta_{X_{i}^{(n)}}$ converges to the measure
$\rho$.

\begin{theorem}[Point process at the edge, neutral regime with fixed background]
	\label{th:edge:neutral:pp}
	Suppose that
	\begin{itemize}
		\item	 $\beta>0$ and fixed;
		\item $\alpha_n$ is such that
		$\lim_{n\to\infty}(\alpha_n-(n-1))=2\lambda\in(0,\infty)$;
		\item $\rho$ is supported inside $(-\infty,0]$ and the support contains
		the origin $0$.
	\end{itemize}
    If $\mathbf{X}^{(n)} \sim P_n$ as in \eqref{eq:Xn}, then for all $m\ge 1$
    the order statistics satisfy
	\begin{align}\label{eq:unconditioned limit}
	\left(X^{(n)}_{(m)},\dots,
	X^{(n)}_{(1)} \right)
	\underset{n\to\infty}{\overset{\mathrm{Law}}{\longrightarrow}}
	\Bigr(\frac{2}{\beta}	\sum_{i=m}^\infty 
	\frac{Z_i}{i(2\lambda-1 + i)},
	\ldots,
	\frac{2}{\beta}	\sum_{i=1}^\infty
	\frac{Z_i}{i(2\lambda-1 + i)}\Bigr),
	\end{align}
	where $\{Z_i\}_{i \geq 1}$ is a sequence of independent exponential random
	variables of unit mean. 
\end{theorem}

The shape of $\rho$ plays no role in the above theorem as long as the support
contains the point 0.

\begin{remark}[Point process between a left-side and right-side background]	
  In the spirit of the two-dimensional analysis in
  \cite{raphael-david,raphael-david-et-al}, one may also consider the jellium
  where the support of $\rho$ is contained in $(-\infty,a] \cup [b,\infty)$
  for some real numbers $a$ and $b$ in the support of $\rho$ with $a<b$. Set
  $\lambda_n = [\alpha_n - (n-1)]/2$, and choose $\eta \in \mathbb C$ with
  $|\eta|=1$. Suppose that $\alpha_n$ satisfies
  \[
	\frac{\alpha_n}{n}
	\xrightarrow[n \to \infty]{} 1
	\quad\text{ and }\quad
	\exp(2\pi \ii (\lambda_n - \alpha_n \rho(-\infty,a])) = \eta.
  \]
  Then using similar arguments to those in this work, we could show that
  $\{X_1^{(n)},\dots,X_n^{(n)}\}\cap[a,b]$ converges to a point process
  depending only on the parameter $\eta$, as in \cite{aizenman1980structure}.
  Here $\eta$ parametrizes the possible limit point processes.
\end{remark}

We finally consider the single right-most particle of the gas:
\begin{equation}\label{eq:munmn}
  M_n=\max_{1\leq i\leq n}X^{(n)}_i=X^{(n)}_{(1)}.
\end{equation}

\begin{corollary}[Tail asymptotics at the right]\label{cor:tail}\ %
  \begin{itemize}
  \item Under the assumptions of Theorem
    \ref{th:edge pp neutral} or  Theorem \ref{th:edge:neutral:pp}, 
    \begin{equation}\label{eq:exptails nlim}
      \lim_{n\to\infty}\mathbb{P}(M_n>t)= \ee^{-\beta\lambda t(1+c_t)}
      \ \text{for all $t$, with}\ 
      \lim_{t\to\infty}c_t=0.
    \end{equation}
  \item Under the assumptions of Theorem \ref{th:edge:nn:pp}, 
	\begin{equation}\label{eq:twtails}
      \lim_{n\to\infty}\mathbb{P}\left(M_n>t\right)%
      =\ee^{-\frac{\beta}{\gamma} t^{\gamma}(1+c_t)} 
      \ \text{for all $t$, with} \lim_{t\to\infty}c_t=0.
	\end{equation}
  \end{itemize}
\end{corollary}


\section{Proofs of main results}\label{sec:proofs}
  
\subsection{Infinite Coulomb gases focused at an edge}\label{sec:icg}

As a preliminary to the proofs, in this subsection, we give a description of
the limiting objects in Theorems \ref{th:edge pp neutral},
\ref{th:edge:nn:pp}, \ref{th:edge:neutral:pp} as Coulomb gases with an
infinite number of particles. By \eqref{eq:Baxter} and in the spirit of
\eqref{eq:PnG}, this will be related to the notion of conditioning an infinite
number of particles to have a specific ordering.

Let $\mu$ be a locally finite measure on $\mathbb R$ such that
$\mu(-\infty,t] = \infty$ for every $t \in \mathbb R$. The notion of an
infinite Coulomb gas on $\mathbb R$ associated to $\mu$ at inverse temperature
$\beta$ is as follows. Take $V$ such that $\Delta V = \mu$. Since
$\mu(-\infty,t] = \infty$ for every $t \in \mathbb R$, we know that
$\lim_{x \to - \infty}V(x)/|x| = \infty$. Since $V$ is convex, by possibly
adding a linear term to $V$, we may assume that there exists
$t_0 \in \mathbb R$ such that $V|_{[t_0,\infty)}$ is non-decreasing. Let
$\lambda > 0$ and $\{Y_k\}_{k \geq 1}$ be independent random variables such
that $Y_k$ has a density proportional to
\[
\exp \big(\!-\beta 
\left[ \left(k - 1 +\lambda \right)x + V(x) \right]\big).
\]
Take a random vector $(\xi^{(n)}_1,\dots,\xi^{(n)}_n)$ such that
\begin{align}\label{def:xi}
\mathrm{Law} (\xi^{(n)}_{(n)},\dots,\xi^{(n)}_{(1)}) = \mathrm{Law}
(Y_n,\dots, Y_1\mid Y_n \leq \dots \leq Y_1 ).
\end{align}
We will be interested in the limit, in $n$, of the point processes
\[
\{\xi^{(n)}_1,\dots,\xi^{(n)}_n\}.
\]
We see this limit as an infinite Coulomb gas at inverse temperature $\beta$
since if $x_1 \geq \dots \geq x_n$ then
\[
\sum_{i=1}^n
\left[\left(i - 1 +\lambda
\right)x_i
+ V(x_i) \right]
=-\frac{1}{2}
\sum_{i<j}^n (x_i-x_j)
+ \sum_{i=1}^n
\Bigr(V(x_i) + \frac{2\lambda + n-1}{2}x_i \Bigr).
\]
Hence the total potential energy contains both a two-body Coulomb interaction
portion as well as a confining potential given by
$V(x) + \frac{2\lambda + n-1}{2}x$, whose Laplacian is $\mu$. By this
construction we obtain a family of infinite Coulomb gases indexed by
$\lambda> 0$.

We will be interested here in the following three cases.
\begin{itemize}
\item Case where for some $\gamma > 1$
  \begin{equation}\label{eq:SquaredGammaPotential}
	V(x) = 
	\frac{x^2}{2}\mathbf{1}_{x<0}
	+\frac{x^\gamma}{\gamma}\mathbf{1}_{x\geq0}
  \end{equation}	 
\item Case 
  \begin{equation}\label{eq:SquaredZeroPotential}
	V(x) = \frac{x^2}{2}\mathbf{1}_{x<0}
  \end{equation}	 
\item Infinite half-well case
  \begin{equation}\label{eq:LeftHardEdgePotential}
	V(x) =
	\begin{cases}
      \displaystyle \infty, & \text{if $x<0$}\\
      \displaystyle 0, 
      & \text{if $x\geq0$}
	\end{cases}.
  \end{equation}	 
\end{itemize}

The infinite half-well (or hard wall) case \eqref{eq:LeftHardEdgePotential}
admits the following simple description.

\begin{proposition}[Infinite Coulomb gas in an infinite half-well]
	\label{prop:NoBackgroundInfiniteCoulombGas}
	Let $\lambda > 0$ and $\beta>0$, and take a sequence $\{Y_k\}_{k \geq 1}$
	of independent random variables such that $Y_k$ has a density proportional
	to
	\[
      x \mapsto \exp \big( \!-\beta \left(k - 1 +\lambda \right)x
      \big)\mathbf{1}_{x \geq 0}.
    \]
    If $\{Z_i\}_{i \geq 1}$ is a sequence of independent exponential random
    variables of unit mean, then
    \[\lim_{n\to\infty}\mathrm{Law} (Y_k\mid Y_n \leq \dots \leq Y_k )=
      \mathrm{Law}\left( \frac{2}{\beta} \sum_{i=k}^\infty
        \frac{Z_i}{i(2\lambda-1 + i)}. \right)
    \]
    Moreover, we have for all $m\geq1$,
    \begin{align}\label{eq:UnchargedLimit}
      \lim_{n\to\infty}    \mathrm{Law}(Y_m,\dots, Y_1\mid Y_n \leq Y_{n-1}\leq \dots \leq Y_1 )
      =
      \mathrm{Law}\Bigr(\frac{2}{\beta}	\sum_{i=m}^\infty 
      \frac{Z_i}{i(2\lambda-1 + i)},
      \dots,
      \frac{2}{\beta}	\sum_{i=1}^\infty
      \frac{Z_i}{i(2\lambda-1 + i)}\Bigr).
    \end{align}
\end{proposition}


As $\lambda\to0$, we lose a particle to infinity, and we recover then the
point process for $\lambda=1$.

\begin{proof}
	Let $k$ be fixed. Let $ T_k^{(n)},\dots, T_n^{(n)}$ be random variables such
	that
	\[
	\mathrm{Law}(T_n^{(n)},\dots,T_k^{(n)})
	=\mathrm{Law}(Y_n,\dots,Y_k \mid Y_n \leq \dots \leq Y_k).
	\]
	Then $(T_n^{(n)},\dots, T_k^{(n)})$ have a joint density proportional to
	\[
	(x_n,\dots,x_k) 
	\mapsto \exp \Bigr[-\beta 
	\sum_{j=k}^n
	\left(j - 1 +\lambda
	\right)x_j
	\Bigr]\mathbf{1}_{0 \leq x_n  \leq \dots \leq x_k}.
	\]
	We can perform the change of variables
	\begin{equation}\label{eq:ChangeOfVariables}
	z_i = x_i - x_{i+1}\text{ if } 
	i \in \{k,\dots, n-1\}
	\quad \text{ and } \quad 
	z_n = x_n 
	\end{equation}
	or equivalently $x_j = \sum_{i=j}^n z_i$ for any $i \in \{k,\dots, n\}$ to
	obtain a density proportional to
	\[
	(z_n,\dots,z_k) \mapsto \prod_{i=k}^n
	\exp\left[-\frac{\beta}{2}(2\lambda + i-1)i\,  z_i\right]\mathbf{1}_{z_i \geq 0}.
	\]
	This tells us that, if $\{Z_i \}_{i \geq k} $ is a sequence of independent
	exponential random variables of unit mean, we have that the law of
	$\bigr(T_{n}^{(n)},\dots,T_k^{(n)}\bigr)$ is the same as the law of
	\begin{equation}\label{eq:UnchargedFiniteVersion}
	\frac{2}{\beta} \Bigr(
	\frac{Z_n}{n(2\lambda + n-1)},
	\sum_{i=n-1}^n
	\frac{Z_i}{i(2\lambda + i-1)},
	\dots,
	\sum_{i=k}^n
	\frac{Z_i}{i(2\lambda + i-1)}\Bigr).
	\end{equation}
	In particular, 
	\[
	\mathrm{Law}(T_k^{(n)})
	= \mathrm{Law} 
	\Bigr( \frac{2}{\beta}
	\sum_{i=k}^n
	\frac{Z_i}{i(2\lambda + i-1)}\Bigr)
	\]
	so that, by taking $n \to \infty$, we
	obtain that
	\[
	T_k^{(n)}
	\xrightarrow[n \to \infty]{}
	\frac{2}{\beta} \sum_{i=k}^\infty
	\frac{W_i}{i(2\lambda + i-1)}.
	\]
	Finally, if $k=1$,
	\[
	\mathrm{Law}(\xi_m^{(n)},\dots, \xi_1^{(n)})
	=
	\mathrm{Law}(T_m^{(n)}, \dots, T_1^{(n)})
	\]
	and we obtain \eqref{eq:UnchargedLimit} by taking $n \to \infty$ in
	\eqref{eq:UnchargedFiniteVersion}.
\end{proof}

\begin{remark}[Gumbel limit]
	If $\frac{2}{\beta}=\chi=2\lambda-1$, let us show that the final coordinate
	of the right-hand side of \eqref{eq:UnchargedLimit} has a Gumbel limit as
	$\chi\to\infty$. Indeed, this coordinate can be written
	\[
	M_\chi= \sum_{i=1}^\infty \frac{\chi} {k(\chi + k)} Z_k
	\]
	where $\{Z_k\}_{k \geq 1}$ are independent exponential random variables of
	unit mean. It turns out that
	\[
	M_\chi - \mathbb E[M_\chi] %
	\xrightarrow[\chi \to \infty] %
	{\mathrm{law}} G - \gamma
	\quad\text{where}\quad
	\gamma = \lim_{n \to \infty}\Bigr(\sum_{k=1}^n\frac{1}{k}-\log n\Bigr)
	\]
	is the Euler\,--\,Mascheroni constant and $G$ is a standard Gumbel random
	variable. Indeed, we could use characteristic functions or Fourier transform
	and start by noting that for any $u \in \mathbb R$,
	\[
	\mathbb E \left[ \ee^{-\ii u M_\chi}\right] %
	= \prod_{k=1}^\infty%
	\Bigr( \displaystyle 1+ \frac{\ii u \chi}{k(\chi+k)} \Bigr)^{-1}.
	\]
	On the other hand,
	\[
	\prod_{k=1}^\infty
	\exp\Bigr(\frac{\ii u \chi}{k(\chi+k)}\Bigr)\Bigr(1+\frac{\ii u \chi}{k(\chi+k)}\Bigr)^{-1}
	\xrightarrow[\chi \to \infty]{}
	\prod_{k=1}^\infty
	\ee^{\ii\frac{u}{k}}\Bigr(1+\frac{\ii u}{k}\Bigr)
	=
	\ii u
	\ee^{\gamma \ii u}\Gamma(\ii u)
	=
	\ee^{\gamma \ii u}\Gamma(\ii u + 1)
	\]
	where we have used Weierstrass's formula
	$\Gamma(z) =z^{-1}\ee^{-\gamma z} \prod_{k=1}^\infty\ee^{k^{-1}z}(1+k^{-1}z)^{-1}$
	and the identity $z\Gamma(z)=\Gamma(z+1)$. It remains to note that if $Z$ is
	an exponential random variable of unit mean so that $G = - \log(Z)$ is a
	standard Gumbel random variable, then for all $u\in\mathbb{R}$,
	\[
	\mathbb E\left[\ee^{-\ii u G} \right] =\int_0^\infty
	\ee^{-r}\ee^{\ii u \log r} \dd r \int_0^\infty
	\ee^{-r}r^{\ii u} \dd r =\Gamma(\ii u + 1 ).
	\]
\end{remark}

We now state the existence of more general infinite Coulomb gases focused at an edge in the
following proposition where, for simplicity, we take $t_0 = 0$.

\begin{proposition}[Infinite Coulomb gases at the edge]\label{prop:LimitStandardCoulomb}
	Let $V:\mathbb R \to \mathbb R$ be a continuous function such that
	\[
	\lim_{x \to - \infty} \frac{V(x)}{|x|} = \infty
	\quad\text{and such that}\quad
	V|_{[0,\infty)} \text{ is non-decreasing}.
	\]
	For $\lambda,\beta > 0$ consider $\{Y_k\}_{k \geq 1}$ and
    $(\xi_1^{(n)},\dots,\xi_n^{(n)})$ as in \eqref{def:xi}. Then, for any
    $k \geq 1$, the \nolinebreak limit
	\[
	\Theta_k = 
	\lim_{n \to \infty}
	\mathrm{Law} \left(Y_k \mid Y_n \leq \dots \leq Y_k\right)
	\text{ exists}.
	\]
	Moreover, if we take $\widetilde Y_k \sim \Theta_k$ independent of
	$(Y_1,\dots,Y_{k-1})$ we have that
	\begin{equation}\label{eq:CoulombLimitDefinition}
	\lim_{n \to \infty}
	\mathrm {Law}
	(\xi_{(k)}^{(n)},\dots,\xi_{(1)}^{(n)})
	=\mathrm {Law}
	(\widetilde Y_k, Y_{k-1},\ldots, Y_1\mid\widetilde Y_k\leq Y_{k-1}\leq \dots \leq  Y_1)
	\end{equation}
\end{proposition}


\begin{proof}
	Let $k$ be fixed and $\{\widetilde Y_k^{(n)}\}_{k \geq n}$ be a sequence of
	random variables such that
	\[
	\mathrm{Law}(\widetilde Y_k^{(n)})
	=\mathrm{Law} \left(Y_k\mid Y_n \leq \dots \leq Y_k
	\right).
	\]
	By Proposition \ref{prop:PushingToTheRight}, the sequence
	$\{\widetilde Y_{k}^{(n)}\}_{n \geq 1}$ is stochastically increasing, namely
	\begin{align}\label{eq:si}
	\mathbb P(\widetilde Y_k^{(n+1)}
	\leq t) 
	\leq \mathbb P(\widetilde Y_k^{(n)} 
	\leq t)
	\end{align}
	for every $t \in \mathbb R$ and $n \geq 1$. It is enough to control
	$\widetilde Y_k^{(n)}$ from above to know that it converges in law. We begin
	by noticing that, again by Proposition \ref{prop:PushingToTheRight},
	\[
	\mathbb P (Y_k \leq t\mid 0\leq Y_n \leq \dots \leq Y_k ) %
	\leq \mathbb P (Y_k \leq t\mid Y_n \leq \dots \leq Y_k ).
	\]
	But $\mathrm {Law} (Y_k\mid 0\leq Y_n \leq \dots \leq Y_k )$ can be
	equivalently described by taking independent random variables
	$W_n, \dots, W_k$ such that $W_i$ has a density proportional to
	\[
	x \mapsto \exp \big(\!-\beta 
	\left[ \left(i - 1 +\lambda
	\big)x
	+ V(x)
	\right]\right)\mathbf{1}_{x \geq 0}
	\]
	and noticing that
	\[
	\mathrm {Law}
	(Y_k\mid 0\leq Y_n \leq \dots \leq Y_k )
	= \mathrm{Law}(W_k\mid W_n \leq \dots \leq W_k ).
	\]
	Let $\{B_i\}_{i \geq k}$ be a sequence of independent random variables such
	that the random variable $B_i$ has a density proportional to
	\[
	x \mapsto \exp \left(-\beta \big(i - 1 +\lambda \right)x
	\big)\mathbf{1}_{x \geq 0}.
	\]
	By Proposition 
	\ref{prop:Hereditary}, we know that
	\[
	\mathbb P(B_k \leq t\mid B_n \leq \dots \leq B_k )
	\leq 
	\mathbb P(W_k \leq t\mid W_n \leq \dots \leq W_k ).
	\] 
	Let $\varepsilon > 0$. By Proposition
	\ref{prop:NoBackgroundInfiniteCoulombGas},
	$\mathrm{Law}(B_k\mid B_n \leq \dots \leq B_k )$ converges as $n \to \infty$
	so that there exists $T>0$ such that
	\[
	1-\varepsilon < 
	\mathbb P(B_k \leq T\mid B_n \leq \dots \leq B_k )
	\]
	for every $n$. This implies that the sequence
	$\{\widetilde Y_k^{(n)}\}_{n \geq k}$ is tight and, since \eqref{eq:si}
	shows it is stochastically increasing, it has a limit that we shall call
	$\widetilde Y_k$.
	
    The right-hand side of \eqref{eq:CoulombLimitDefinition} is well-defined
    since we may find open intervals $I_1,\dots,I_{k}$ such that
	\[
	\mathbb P \left(
	\widetilde Y_k 
	\leq Y_{k-1} \leq \dots \leq
	Y_{1}  \right) 
	\geq 
	\mathbb P\left(
	\widetilde Y_k \in I_k,
	Y_{k-1}  \in I_{k-1}, \dots,
	Y_1 \in I_1
	\right)
	\, > \, 0.
	\]
	To prove the second statement of the proposition we notice that, by Lemma
	\ref{lem:ConditionBySteps},
	\[
	\mathrm{Law}
	\left(\xi_k^{(n)},\dots,
	\xi_1^{(n)}\right)
	=  \mathrm{Law}
	\left(\widetilde Y_k^{(n)},
	Y_{k-1},\dots,Y_{1} \bigm\vert
	\widetilde Y_k^{(n)} \leq 
	Y_{k-1} \leq \dots
	\leq Y_1
	\right)
	\] 
	where we are supposing $\widetilde Y_k^{(n)}$ is independent of
	$(Y_1,\dots,Y_{k-1})$. In particular, we have that
	\[
	\lim_{n \to \infty}
	\mathbb P
	\left(
	\widetilde Y_k^{(n)}
	\leq Y_{k-1} \leq \dots 
	\leq Y_1\right)
	=
	\mathbb P
	\left(
	\widetilde Y_k
	\leq Y_{k-1} \leq \dots \leq Y_{1}
	\right)
	\]
	since $Y_1,\dots,Y_{k-1}$ do not have atoms and thus
	\[
	\mathbb P
	\left(Y_i= Y_j\right) = 
	0 \quad \text{ for }
	i \neq j \quad \text{ and }
	\quad
	\mathbb P\left(Y_i=
	\widetilde Y_k\right) = 0
	\quad \text{ for } i \in
	\{1,\dots, k-1\}.
	\]
	Finally, for any closed set $A \subset \mathbb R^k$,
	\begin{align*}
	\limsup_{n \to \infty}
	& \
	\mathbb P\Big((Y_1,\dots,Y_{k-1},
	\widetilde Y_k^{(n)}) \in  A \quad
	\text{ and } \quad 
	\widetilde Y_k^{(n)} \leq
	Y_{k-1} \leq \dots
	\leq Y_{1} 
	\Big)								\\
	&\leq
	\mathbb P\Big((Y_1,\dots,Y_{k-1},
	\widetilde Y_k) \in  A \quad
	\text{ and } \quad 
	\widetilde Y_k \leq
	Y_{k-1} \leq \dots
	\leq Y_{1} 
	\Big),
	\end{align*}
	so that
	\begin{align*}
	\limsup_{n \to \infty}
	& \
	\frac{\mathbb P\Big((Y_1,\dots,Y_{k-1},
		\widetilde Y_k^{(n)}) \in  A \quad
		\text{ and } \quad 
		\widetilde Y_k^{(n)} \leq
		Y_{k-1} \leq \dots
		\leq Y_{1} 
		\Big)}
	{\mathbb P
		\left(\widetilde Y_k^{(n)}
		\leq Y_{k-1} \leq \dots \leq Y_{1}
		\right)}								\\
	&\leq
	\frac{\mathbb P\Big((Y_1,\dots,Y_{k-1},
		\widetilde Y_k) \in  A \quad
		\text{ and } \quad 
		\widetilde Y_k \leq 
		Y_{k-1} \leq \dots
		\leq Y_{1}
		\Big)}{\mathbb P
		\left(
		\widetilde Y_k \leq 
		Y_{k-1} \leq \dots \leq Y_{1} 
		\right)}.
	\end{align*}
	This is one of the equivalent conditions of weak convergence, thus the proof
	is complete.
\end{proof}

\begin{remark}[Limiting point process]
	The convergence \eqref{eq:CoulombLimitDefinition} defines a probability
	measure on the space of sequences of real numbers
	$\mathbb R^{\mathbb Z_{\geq 1}}$. More precisely, if we let, for any
	$k \geq 1$, $\pi_k: \mathbb R^{\mathbb Z_{\geq 1}} \to \mathbb R^k$ be the
	projection onto the first $k$ coordinates and
	$Y_1,\dots,Y_{k-1},\widetilde Y_k$ be as in Proposition
	\ref{prop:LimitStandardCoulomb}, then the infinite Coulomb gas is the
	probability measure $\Gamma$ on $\mathbb R^{\mathbb Z_{\geq 1}}$ such that,
	for every $k \geq 1$,
	\[
	{(\pi_k)}_*\Gamma %
	=\mathrm{Law}(\widetilde Y_k,Y_{k-1},\dots,Y_1\mid\widetilde Y_k \leq Y_{k-1} \leq \dots \leq Y_1),
	\]
	where ${(\pi_k)}_* \Gamma $ denotes the image measure of $\Gamma$ by
	$\pi_k$. In particular, for any continuous $f:\mathbb R \to \mathbb R$ whose
	support is bounded from below, we have
	that
	\[
	\sum_{i=1}^n
	f(\xi_i^{(n)}) 
	\xrightarrow[n \to \infty]{\mathrm{Law}}
	\sum_{i=1}^\infty f(W_i)
	\quad\text{where}\quad
	(W_i)_{i \geq 1} \sim \Gamma.
	\]
	This statement is in fact equivalent to \eqref{eq:CoulombLimitDefinition} for every $k \geq 1$.
\end{remark}

\subsection{Proofs of Theorems \ref{th:edge:neutral:pp},
  \ref{th:edge pp neutral}, \ref{th:edge:nn:pp}, Corollary
  \ref{cor:tail}}\label{sec:pp:proofs}

The following proofs share the same first steps which we explain now. We
define
\begin{equation}\label{eq:ithPotential}
  V^{(n)}_i(x) = \frac{2i-1-n}{2}x-\alpha_nU_{\rho_n}(x).
\end{equation}
For every $n\geq 1$, we consider $n$ independent random variables
$Y_1^{(n)},\dots,Y_n^{(n)}$ such that $Y_i^{(n)}$ has a density proportional
to $x \in \mathbb R \mapsto \ee^{-\beta V^{(n)}_i (x)}$. By
\eqref{eq:CombinatorialIdentity}, the Coulomb gas $X_1^{(n)},\dots,X_n^{(n)}$
satisfies
\[
  \mathrm{Law}
  \left(X_{(n)}^{(n)},
    \dots,X_{(1)}^{(n)} \right)
  =
  \mathrm{Law}\left(Y_n^{(n)},\dots,Y_1^{(n)}
    \bigm\vert
    Y_n^{(n)} \leq \dots \leq Y_1^{(n)}\right).
\]
We fix $k \geq 1$ and we study
$\left(X_{(k)}^{(n)}, \dots,X_{(1)}^{(n)} \right)$. By Lemma
\ref{lem:ConditionBySteps}, there is a simple description of this vector by
considering a random variable $\widetilde Y_k^{(n)}$ independent of
$Y_1^{(n)},\dots,Y_{k-1}^{(n)}$ and such that
\[
  \mathrm{Law}(\widetilde Y_k^{(n)})
  = \mathrm{Law}
  \left(Y^{(n)}_k \bigm\vert Y^{(n)}_n 
    \leq  \dots \leq Y^{(n)}_k\right).
\]
Namely, we have that
\[
  \mathrm{Law}
  \left(X_{(k)}^{(n)},
    \dots,X_{(1)}^{(n)} \right)
  =
  \mathrm{Law}\left(\widetilde Y_k^{(n)},
    Y_{k-1}^{(n)},\dots,Y_1^{(n)}
    \bigm\vert
    \widetilde Y_k^{(n)} \leq 
    Y_{k-1}^{(n)} \leq \dots \leq Y_1^{(n)}\right).
\]
This suggests dividing the argument into two parts: first, understand the
limiting law of the random vector
$ \left(\widetilde Y_k^{(n)}, Y_{k-1}^{(n)},\dots,Y_1^{(n)} \right)$ and next,
perform the conditioning. Since
$\widetilde Y_k^{(n)}, Y_{k-1}^{(n)},\dots,Y_1^{(n)}$ are independent, we need
only understand the limit of each one separately. Our goal is to obtain the
infinite Coulomb gas discussed in Section \ref{sec:icg}. We first take the
limits
\begin{equation}\label{eq:NotconditionedConvergence}
  Y_i^{(n)} \xrightarrow[n \to \infty]{\mathrm{law}} Y_i,
\end{equation}
where the law of $Y_i$ is proportional to
\[
  \exp\big(\!-\beta\left[ \left(
      \lambda + i - 1
    \right)x + V(x)\right] \big)
  \dd x,
\]
and $V$ is one of the potentials in \eqref{eq:SquaredGammaPotential},
\eqref{eq:SquaredZeroPotential} or \eqref{eq:LeftHardEdgePotential}. Next, we
take the limits
\begin{equation}\label{eq:kLargestConvergence}
  \widetilde Y_k^{(n)} \xrightarrow[n \to \infty]
  {\mathrm{Law}} \widetilde Y_k,
\end{equation}
where the law of $\widetilde Y_k$ is the limit as $m\to\infty$ of
\[
  \mathrm{Law}\left( Y_k\mid Y_m \leq \dots \leq Y_k \right)
\]
which exists due to Proposition \ref{prop:LimitStandardCoulomb} and
\ref{prop:NoBackgroundInfiniteCoulombGas}. After proving
\eqref{eq:NotconditionedConvergence} and \eqref{eq:kLargestConvergence}, a
standard conditioning argument such as the one at the end of the proof of
Proposition \ref{prop:LimitStandardCoulomb}, completes the proof. We proceed
now to prove \eqref{eq:NotconditionedConvergence} and
\eqref{eq:kLargestConvergence} for each of the cases. We begin with Theorem
\ref{th:edge:neutral:pp}.
 
\begin{proof}[Proof of Theorem \ref{th:edge:neutral:pp}]
  We begin by noticing that, by \eqref{eq:PotentialExplicitFormula},
  \begin{align*}
    V^{(n)}_i(x)
    &=
      \frac{2i-1-n}{2}x-\alpha_nU_\rho(x)\\
    &=
      \left(i-1 + \frac{\alpha_n - (n-1) }{2}\right)x
      + \alpha_n\int_{(x,0]} (s-x)\dd\rho(s) 
      - \frac{\alpha_n}{2}\int_{\mathbb R}s\dd\rho(s).
  \end{align*}
  Since $\frac{\alpha_n}{2} \int_{\mathbb R} s \, \dd \rho(s)$ is independent
  of $x$ we may redefine
  \[
    V^{(n)}_i(x)
    =
    \left(i-1 + \frac{\alpha_n - (n-1) }{2}\right)x
    + \alpha_n 
	\int_{(x,0]} (s-x)
    \dd \rho(s).
  \]
  In this case $V$ will be the one in \eqref{eq:LeftHardEdgePotential} so that
  $Y_i$ is proportional to
  \[
    \exp\big(\!-\beta (
    i - 1 + \lambda
    )x\big)
    \mathbf{1}_{[0,\infty)}(x)\dd x,
  \]
  To get \eqref{eq:NotconditionedConvergence} it is enough to notice that
  \[
    \lim_{n \to \infty} V_i^{(n)}(x)
    =\begin{cases}
       \big( i-1 + \lambda \big)x
       &\text{ if $x \geq 0$}	
       \infty\\
       &\text{ if $x < 0$}
     \end{cases}.
   \]
   so that, since for any $\varepsilon \in (0, \lambda)$ and $N > i$
   \[
     V_i^{(n)}(x) > 
     \big(i-1 + \lambda - \varepsilon\big)x
     + N
     \Bigr(\int_{(x,0]} s\dd \mu(s) - P(x)\Bigr)
   \]
   for $n$ large enough, dominated convergence gives
   \eqref{eq:NotconditionedConvergence}. What is left to prove is
   \eqref{eq:kLargestConvergence}. By Proposition
   \eqref{prop:NoBackgroundInfiniteCoulombGas}, the limit of
   $\mathrm{Law}\left( Y_k\mid Y_m \leq \dots \leq Y_k \right)$ as
   $m \to \infty$ is
   \[
	   \frac{2}{\beta} \sum_{j=k}^\infty\frac{Z_j}{j(2\lambda - 1 + j)},
   \]
   where $\{Z_j\}_{j \geq k}$ is a sequence of independent exponential random
   variables of unit mean. Now, to study the sequence
   $\{\widetilde Y_k^{(n)}\}_{ n \geq k}$, we shall consider two new random
   variables. The first one will be an upper bound for $\widetilde Y_k^{(n)}$
   and will be denoted by $U^{(n)}$. It is defined so that it satisfies
   \[
     \mathrm{Law}(U^{(n)})
     = \mathrm{Law}
     \left(Y^{(n)}_k \bigm\vert
         0 \leq Y^{(n)}_n 
         \leq  \dots \leq Y^{(n)}_k\right).
   \]
   For the second one we fix an integer $M>k$. The random variable will be a
   lower bound of $\widetilde Y_k^{(n)}$ and will be denoted by $L^{(n)}_M$.
   We ask $L^{(n)}_M$ to satisfy
   \[
     \mathrm{Law}(L_M^{(n)}) = \mathrm{Law} \left(Y^{(n)}_k \bigm\vert
         Y^{(n)}_M \leq \dots \leq Y^{(n)}_k\right).
   \]
   By Proposition \ref{prop:PushingToTheRight}, we have the stochastic
   domination
   \[
     \mathbb P \left( U^{(n)}
       \leq t \right)
     \leq \mathbb P \left( \widetilde Y^{(n)}_k
       \leq t \right)
     \leq \mathbb P \left(  L_M^{(n)}
       \leq t \right).
   \]
   To understand $U^{(n)}$ we notice that
   the density of
   \[
     \mathrm{Law}
     \left(Y^{(n)}_n, \dots, Y^{(n)}_k \bigm\vert
         0 \leq Y^{(n)}_n 
         \leq  \dots \leq Y^{(n)}_k\right)
   \]
   is proportional to
   \[
     (y_n,\dots,y_k) \mapsto \prod_{i=k}^n
     \exp\left[-\frac{\beta}{2}(\alpha_n - (n-1) + 2(i - 1))y_i\right] 
     \mathbf{1}_{0 \leq y_{n} \leq \dots \leq y_k}.
   \]
   As in the proof of Proposition \ref{prop:NoBackgroundInfiniteCoulombGas},
   under the change of variables \eqref{eq:ChangeOfVariables}, we can see that
   \[
     \mathrm{Law}(U^{(n)})
     = \mathrm{Law}\Bigr(
     \frac{2}{\beta}
     \sum_{j=k}^n\frac{Z_j}{j(\alpha_n-(n-1) + j-1)} \Bigr).
   \]
   Since $\alpha_n - (n-1) \xrightarrow[n\to \infty]{} 2\lambda > 0$, we can
   bound each term in the sum by $\frac{Z_j}{j(2\lambda + j-1)}$ for $n$ large
   enough. Since $ \sum_{j\ge 1}\frac{Z_j}{j(\lambda+j-1)}$ converges almost
   surely and since $ \frac{Z_j}{j(\alpha_n - (n-1) + j-1)}$ converges to
   $\frac{Z_j}{j(2\lambda + j-1)}$, we can use the dominated convergence
   theorem to obtain that
   \[
	   \frac{2}{\beta} \sum_{j=k}^n\frac{Z_j}{j(\alpha_n - (n-1) + j-1)} \xrightarrow[n \to \infty]{\mathrm{a.s.}} \frac{2}{\beta}
     \sum_{j=k}^\infty\frac{Z_j}{j(2\lambda + j-1)}. 
   \] 
   To understand $L_M^{(n)}$ we take the limit of $( Y^{(n)}_M,\dots,Y^{(n)}_k) $
   which is $(Y_M,\dots,Y_k)$ by \eqref{eq:NotconditionedConvergence}. Since
   the random variables involved have no atoms, taking the limit commutes with
   conditioning so that
   \[
     L_M^{(n)}
     \xrightarrow[n \to \infty]{\mathrm{Law}}
     L_{M},
   \]
   where
   \[
     \mathrm{Law}(L_{M})=\mathrm{Law}\left( Y_k \mid Y_M \leq \dots \leq Y_k
     \right).
   \]
   As before, we can see that
   \[
     \mathrm{Law}(L_{M}) = 
     \mathrm{Law}\Bigr(\sum_{j=k}^M\frac{Z_j}{j(\chi +(j-1))}\Bigr),
   \]
   where $\{Z_j\}_{j \geq k} $ is a sequence of independent exponential random
   variables of unit mean. By taking upper and lower limits in
   \[
     \mathbb P( U^{(n)} \leq t)
     \leq \mathbb P( \widetilde Y^{(n)}_k \leq t)
     \leq \mathbb P( L_M^{(n)} \leq t),
   \]
   we get that
   \begin{multline*}
	   \mathbb P\Bigr(\frac{2}{\beta} \sum_{j=k}^\infty\frac{W_j}{j(\chi + (j-1))} \leq t \Bigr)					\\
       \leq 
       \liminf_{n \to \infty} 
       \mathbb P(\widetilde Y^{(n)}_k\leq t) 
       \leq \limsup_{n \to \infty}\mathbb P(\widetilde Y^{(n)}_k\leq t)\\
       \leq
	   \mathbb P\Bigr(\frac{2}{\beta} \sum_{j=k}^M \frac{W_j}{j(\chi + (j-1))} \leq t\Bigr).
   \end{multline*}
   By taking $M \to \infty$ we obtain \eqref{eq:kLargestConvergence}.
 \end{proof}
 
\begin{proof}[Proof of Theorem \ref{th:edge pp neutral}]
  We have $\alpha_n\rho_n(x)=\ind_{[-\alpha_n,0]}(x)$. Since
  $\alpha_n \to \infty$, every subsequence has an increasing (sub)subsequence
  so that we may assume without loss of generality that $\alpha_n$ is
  increasing. Let us define
  \[
    V(x) = \frac{x^2}{2}\mathbf{1}_{x\leq0}.
  \]
  We have $\Delta V(x) = \mathbf{1}_{(-\infty,0]}(x)$ for $x\neq0$. To
  understand $V_i^{(n)}$, which is
  $\frac{2i-n-1}{2}x- \alpha_n U_{\rho_n}(x)$ by
  \eqref{eq:ithPotential}, we notice that
  $-\alpha_n U_{\rho_n}(x) = V(x) + \frac{\alpha_n}{2} x$ if
  $x\in\left[-\alpha_n, \infty \right)$, and for $x \in (-\infty,-\alpha_n]$,
  it is affine such that it is differentiable everywhere. One may see this by
  direct calculation using \eqref{eq:PotentialExplicitFormula}, or by noticing
  that the Laplacians of both $-\alpha_n U_{\rho_n}$ and $V$ are the same
  inside $[-\alpha_n,\infty)$, and then by calculating the derivative of
  $-\alpha_n U_{\rho_n}$ at $-\alpha_n$. Therefore, setting
  \[
    V_i(x)=
    \Bigr(i-1+\frac{\alpha_n-(n-1)}{2}\Bigr)x + V(x)
    =
    \left(i-1   + \lambda \right)x + V(x)
  \]
  we have that
  \[
    V^{(n)}_i(x)
    =V_i(x)
    \quad  \text{ if } x
    \in 
    \left[-\alpha_n, \infty
    \right)
  \]
  and that it is extended in an affine and differentiable way at the left of
  this interval. We also have that $Y_i$ has a density proportional to
  $\ee^{-\beta V_i(x)}$. By dominated convergence, since
  $\{V_i^{(n)}\}_{n \geq i}$ is an increasing sequence of functions that
  converges to $V_i$, we obtain \eqref{eq:NotconditionedConvergence}.

  We now prove \eqref{eq:kLargestConvergence}. For this, we shall consider two
  new random variables. The first one will bound $\widetilde Y_k^{(n)}$ from
  above and will be denoted by $U^{(n)}$. It is defined so that it satisfies
  \[
    \mathrm{Law}(U^{(n)}) =
    \mathrm{Law} \left(Y_k \mid Y_n \leq \dots \leq Y_k \right).
  \]
  For the second one we fix an integer $M>k$. The random variable will be a
  lower bound of $\widetilde Y_k^{(n)}$ and will be denoted by $L^{(n)}_M$. We
  ask $L^{(n)}_M$ to satisfy
  \[
    \mathrm{Law}(L_M^{(n)})
    = \mathrm{Law}(Y^{(n)}_k \mid Y^{(n)}_M \leq  \dots \leq Y^{(n)}_k).
  \]
  By Lemma \ref{lem:HereditaryPotential}, we have the stochastic domination
  \[
    \mathbb P( U^{(n)} \leq t)
    \leq \mathbb P(\widetilde Y^{(n)}_k\leq t)
  \]
  while, by Proposition \ref{prop:PushingToTheRight},
  \[
    \mathbb P(\widetilde Y^{(n)}_k\leq t)
    \leq \mathbb P(L_M^{(n)}\leq t).
  \]
  By definition,
  \[
    U^{(n)}\xrightarrow[n \to \infty]{\mathrm{law}} 
    \widetilde Y_k,
  \]
  so that we also have
  \[
    \lim_{n \to \infty} \mathbb P(U^{(n)} \leq t)
    = \mathbb P(\widetilde Y_k \leq t)
  \]
  when $t$ is a point of continuity of $\mathbb P(\widetilde Y_k \leq t)$.
  Since, we also have
  \[
    L_M^{(n)} \xrightarrow[n \to \infty]{\mathrm{law}}
    U^{(M)},
  \]
  we obtain that
  \[
    \mathbb P(\widetilde Y_k\leq t)
    \leq \liminf_{n \to \infty}\mathbb P(\widetilde Y^{(n)}_k\leq t)
    \leq 
    \limsup_{n \to \infty}\mathbb P(\widetilde Y^{(n)}_k\leq t)
    \leq
    \mathbb P(U^{(M)}\leq t).
  \]
  Taking $M \to \infty$ proves \eqref{eq:kLargestConvergence} thus completing
  the proof.
\end{proof}

\begin{proof}[Proof of Theorem \ref{th:edge:nn:pp}]
  We have that
  \[
    \alpha_n\rho_n(x)
	=\ind_{[-\frac{\alpha_n+n}2,0]}(x)
	+(\gamma-1)x^{\gamma-2}\ind_{\left(0,(\frac{\alpha_n-n}2)^{1/(\gamma-1)}\right]}(x).
  \]
  Since $\alpha_n \to \infty$, every subsequence has an increasing
  (sub)subsequence so that we may assume $\alpha_n$ is increasing. Let us define
  \[
    V(x) = \frac{x^2}{2}\mathbf{1}_{x\leq0}
    +\frac{x^\gamma}{\gamma}\mathbf{1}_{x>0}
  \]
  whose Laplacian is
  \[
    \Delta V(x) = \mathbf{1}_{(-\infty,0]}(x)
    +(\gamma-1)x^{\gamma-2}\mathbf{1}_{(0,\infty)}.
  \]
  Notice that
  \[
    -\alpha_n U_{\rho_n}(x)
    = V(x) + \frac{n}{2} x \quad  \text{ if } x
    \in 
    \Bigr[-\frac{\alpha_n+n}2,\Bigr(\frac{\alpha_n-n}2\Bigr)^{1/(\gamma-1)}\Bigr]
  \]
  and that it is affine at the left and at the right of this interval with the
  derivatives at the endpoints coinciding. One may obtain this by a direct
  calculation using \eqref{eq:PotentialExplicitFormula}, or by noticing that
  the Laplacians of both functions, $-\alpha_n U_{\rho_n}$ and $V$, are the
  same inside that interval and the derivatives of $-\alpha_n U_{\rho_n}$
  should be $-\alpha_n/2$ at the left and $\alpha_n/2$ at the right of the
  interval. If we define
  \[
    V_i(x)=
    \left(i   - \frac{1}{2} \right)x + V(x)
  \]
  we have that
  \[
    V^{(n)}_i(x)
    =V_i(x)
    \quad  \text{ if } x
    \in 
    \Bigr[-\frac{\alpha_n+n}2,\Bigr(\frac{\alpha_n-n}2\Bigr)^{1/(\gamma-1)}\Bigr]
  \]
  and that it is extended in an affine and differentiable way outside this
  interval. In this case we have the law of $Y_i$ is proportional to
  \[
    x \mapsto\ee^{-\beta V_i(x)}.
  \]
  By dominated convergence, since $\{V_i^{(n)}\}_{n \geq i}$ is an increasing
  sequence of functions that converges to $V_i$, we obtain
  \eqref{eq:NotconditionedConvergence}, namely
  \[
    Y_{i}^{(n)} \xrightarrow[n \to \infty]
    {\mathrm{law}} Y_i.
  \]
  Now fix $k \geq 1$. To prove \eqref{eq:kLargestConvergence} we consider a
  random vector
  $\left( \widetilde Y^{(n)}_k,\dots,\widetilde Y^{(n)}_n \right)$ such that
  \[
    \mathrm {Law}  \left(\widetilde Y^{(n)}_k,
      \dots,\widetilde Y^{(n)}_n \right)
    = \mathrm{Law} \left(Y_{k}^{(n)}, \dots, Y_{n}^{(n)} \mid Y_n^{(n)}
      \leq \dots \leq Y_k^{(n)} \right).
  \]
  Consider the event
  \[
    A_n= 
    \Bigr\{
      -\frac{\alpha_n+n}2 \leq \widetilde Y_{n}^{(n)} 
      \quad \text{ and } \quad 
      \widetilde Y_{k}^{(n)}  \leq 
      \Bigr(\frac{\alpha_n-n}2 \Bigr)^{1/(\gamma-1)}
    \Bigr\}.
  \]
  We will see that
  \begin{equation}\label{eq:ProbGoes1}
    \mathbb P(A_n) \xrightarrow[n \to \infty]{} 1,
  \end{equation}
  so that, given a bounded continuous function $f: \mathbb R \to \mathbb R$
  and by writing
  \[
    \mathbb E[f(\widetilde Y_k^{(n)})] 
    =\mathbb E[f(\widetilde Y_k^{(n)})\mid A_n]\ \mathbb P(A_n)
    \ +\ 
    \mathbb E[f(\widetilde Y_k^{(n)})\mid A_n^c]\ \mathbb P(A_n^c),
  \]  
  we see that 
  \begin{equation}\label{eq:LimitByConditioning}
    \lim_{n \to \infty}
    \mathbb E[f 
      \bigr(\widetilde Y_k^{(n)}\bigr)]  
    =\lim_{n \to \infty}
    \mathbb E[f\bigr(\widetilde Y_k^{(n)}\bigr)\bigm\vert A_n] 
  \end{equation}
  assuming one of these limits exists. To see that the right-side limit
  exists, we repeat the same argument except with the sequence
  $\{Y_i\}_{i \geq k}$ instead of $\{Y_i^{(n)}\}_{n \geq i \geq k}$. If we
  show the analogue of \eqref{eq:ProbGoes1}, for $\{Y_i\}_{i \geq k}$, then we
  also have the analogue of \eqref{eq:LimitByConditioning}. But note that for
  $\{Y_i\}_{i \geq k}$, the left-hand side of \eqref{eq:LimitByConditioning}
  is known to exist by Proposition \ref{prop:LimitStandardCoulomb} and the
  right-hand side coincides with the one associated to
  $\{Y_i^{(n)}\}_{n \geq i \geq k}$ since $V_i = V_i^{(n)}$ in
  $\left[-(\alpha_n+n)/2,((\alpha_n-n)/2)^{1/(\gamma-1)} \right]$. Thus we
  need only show \eqref{eq:ProbGoes1} with respect to both
  $\{Y_i^{(n)}\}_{n \geq i \geq k}$ and $\{Y_i\}_{i \geq k}$.
 
  By Proposition \ref{prop:PushingToTheRight}, we know that
  \[
    \mathbb P(Y_k^{(n)}\leq t \bigm\vert 0 \leq Y_n^{(n)} \leq \dots \leq
      Y_k^{(n)}) 
      \leq \mathbb P(Y_k^{(n)} \leq t\bigm\vert Y_n^{(n)}
        \leq \dots \leq Y_k^{(n)}).
  \]
  Let $\{B_i\}_{i \geq k}$ be positive
  independent random variables such that $B_i$ has
  a density proportional to
  \[
    x \in (0,\infty) \mapsto \exp \Bigr[ -\beta 
      \Bigr( i-\frac{1}{2} \Bigr) x \Bigr].
  \]
  Since $V_i^{(n)}(x) - (i-1/2)x$ is increasing for 
  $x > 0$, by Lemma
  \ref{lem:HereditaryPotential} we have that
  \[
    \mathbb P(B_k\leq t \bigm\vert 
        B_n \leq \dots \leq B_k)
    \leq 
    \mathbb P(Y_k^{(n)} \leq t\bigm\vert
        0 \leq 
        Y_n^{(n)} \leq \dots \leq Y_k^{(n)}).
  \]
  But, by Proposition \ref{prop:NoBackgroundInfiniteCoulombGas}, we know that
  $\inf_{n \geq k} \mathbb P\left(B_k\leq t \mid B_n \leq \dots \leq
    B_k\right) $ converges to $1$ as $t \to \infty$, so that the same happens
  to $\inf_{n \geq k} \mathbb P\left(\widetilde Y_k^{(n)} \leq t\right) $ and,
  in particular,
  \[
    \mathbb P \Bigr(\widetilde Y_{k}^{(n)}  \leq 
    \Bigr(\frac{\alpha_n-n}2 \Bigr)^{1/(\gamma-1)}
    \Bigr) \xrightarrow[n \to \infty]{} 1.
  \]
  To prove that
  \begin{equation}\label{eq:LeftEdge}
    \mathbb P\Bigr(
      -\frac{\alpha_n+n}2 \leq \widetilde Y_{n}^{(n)} 
    \Bigr) \xrightarrow[ n \to \infty]{} 1
  \end{equation}
  we consider the system seen from $-n$. More precisely, we use that
  \[
    \mathbb P(Y_n^{(n)}\geq t \bigm\vert
        Y_n^{(n)} \leq \dots \leq Y_k^{(n)}
        \leq -n)
    \leq 
    \mathbb P
    (Y_n^{(n)} \geq t\bigm\vert
        Y_n^{(n)} \leq \dots \leq Y_k^{(n)})
  \]
  and we focus on the left-hand side of this inequality that can be rewritten
  \[
    \mathbb P(-n-Y_n^{(n)}\leq -n-t \bigm\vert
        0\leq -n-Y_k^{(n)} \leq \dots \leq -n- Y_n^{(n)}).
  \]
  The same argument as before applied to $-n-Y_k^{(n)},\dots,-n-Y_n^{(n)}$
  instead of $Y_n^{(n)},\dots,Y_k^{(n)}$ gives
  \[
    \mathbb P\Bigr(-n-Y_n^{(n)}\leq 
      \frac{\alpha_n - n}{2} \Bigm\vert
        0\leq -n-Y_k^{(n)} \leq \dots \leq -n- Y_n^{(n)}
      \Bigr) \xrightarrow[n \to \infty]{} 1
  \]
  which implies \eqref{eq:LeftEdge}. The same arguments work for the sequence
  $\{Y_i\}_{i \geq k}$ instead of $\{Y_i^{(n)}\}_{n \geq i \geq k}$ which
  completes the proof.
\end{proof}

\begin{proof}[Proof of Corollary \ref{cor:tail}]
  By Theorems \ref{th:edge:neutral:pp}, \ref{th:edge pp neutral} and
  \ref{th:edge:nn:pp} we know that $M_n$ converges in law. The limiting law is
  $\Theta_1$ from Proposition \ref{prop:LimitStandardCoulomb} by choosing $V$
  properly, and $\Theta_1$ can be also described as
  \[
    \Theta_1 = \mathrm{Law}
    \bigr(Y_1 \bigm\vert \widetilde Y_2 \leq Y_1\bigr)
  \]
  where $\widetilde Y_2$ and $Y_1$ are independent random variables such that
  $\widetilde Y_2 \sim \Theta_2$ and $Y_1$ has a density proportional to
  $\rho(x) = e^{-\beta (\lambda x + V)}$. In particular, $\Theta_1$ has a
  density proportional to
  \[
    x \in \mathbb R \mapsto \rho(x) \mathbb P
    \bigr(\widetilde Y_2 \leq x \bigr).
  \]
  In fact, by repeating the same kind of procedure with $\Theta_k$ instead of
  $\Theta_1$ and using induction on $k$ it can be proved that the regularity
  of the density coincides with the regularity of $V$ but this will not be
  needed here. Since $\Theta_1$ has no atoms we have that
  \[
    \lim_{n\to\infty}\mathbb{P}(M_n>t)
    = \Theta_1(t,\infty)
  \]
  and we may conclude by Proposition \ref{prop:TailAsymptotics}.
\end{proof}

\appendix

\section{Tail asymptotics}\label{se:tail}

In this appendix, we prove a proposition regarding the tail asymptotics of the
right-most particles of our infinite Coulomb gases at the edge. We first need
a short lemma to prove the proposition.

\begin{lemma}\label{lem:ConditionedAsymptotics}
  If $Y_1$ and $Y_2$ are independent random variables, then
  \[
    \mathbb P(Y_1 \geq t, Y_2 \leq Y_1) %
    \underset{t\to\infty}{\sim}
    \mathbb P(Y_1\geq t).
  \]
\end{lemma}

\begin{proof}
  By independence of $Y_1$ and $Y_2$,
  \[
    \mathbb P(Y_1 \geq t , Y_2 \leq Y_1)
    = \int_{\mathbb{R}}\mathbb P(Y_1 \geq t \vee y)
    \dd\mathbb{P}_{Y_2}(y),
  \]
  where $\mathbb{P}_{Y_2}$ denotes the law of $Y_2$. Since
  $ \frac{\mathbb P(Y_1 \geq t \vee y)}{\mathbb P(Y_1 \geq t)}$ is increasing
  in $t$ we may use the monotone convergence theorem to obtain
  \[
    \frac{\mathbb P(Y_1 \geq t , Y_2 \leq Y_1)}
    {\mathbb P(Y_1 \geq t)}
    = \int_{\mathbb{R}}
    \frac{\mathbb P(Y_1 \geq t \vee y)}
    {\mathbb{P}(Y_1\geq t)}
    \dd \mathbb{P}_{Y_2}(y)
    \underset{t\to\infty}{\longrightarrow}
    \int_{\mathbb{R}} \dd \mathbb{P}_{Y_2}(y)=1. 
  \]
\end{proof}

\begin{proposition}[Tail asymptotics at the right]\label{prop:TailAsymptotics}
  Using the notation of Proposition \ref{prop:LimitStandardCoulomb}, if
  $X \sim \Theta_1$ is the right-most particle of the infinite Coulomb gas, we
  have that
  \[
    \log \mathbb P(X \geq t) 
    = \log \int_{t}^\infty\ee^{-\beta \left(\lambda x + V(x)\right)} 
    \dd x + O(1).
  \]
  In particular, if $V$ is \eqref{eq:SquaredGammaPotential} for some
  $\gamma>1$, we have that
  \[
    \frac{1}{t^\gamma}
    \log \mathbb P(X \geq t) 
    \xrightarrow[t \to \infty]{}
    -\frac{\beta}{\gamma},
  \]
  and if $V$ is \eqref{eq:SquaredZeroPotential} or
  \eqref{eq:LeftHardEdgePotential} we have that
  \[
    \frac{1}{t} \log \mathbb P(X \geq t) \xrightarrow[t \to \infty]{} -\beta
    \lambda.
  \]
\end{proposition}

\begin{proof}
  Let $Y_1$ have a density proportional to
  $\exp \big[\!-\beta \left( \lambda x + V(x) \right)\big]$ and let
  $Y_2 \sim \Theta_2$ from Proposition \ref{prop:LimitStandardCoulomb} be
  independent of $Y_1$. By \eqref{eq:CoulombLimitDefinition},
  $
  \mathbb P(X \geq t) %
  = \mathbb P \left(Y_1 \geq t \mid Y_2 \leq Y_1 \right).
  $
  By Lemma \ref{lem:ConditionedAsymptotics}
  \[
    \log \mathbb P(X \geq t) 
    = \log \mathbb P(Y_1 \geq t) + O(1) = \log \int_{t}^\infty\ee^{-\beta
      \left(\lambda x + V(x)\right)} \dd x + O(1).
  \]
  Suppose that $V(x) = \frac{x^\gamma}{\gamma}$, $x \geq 0$, $\gamma>1$.
  Then, for every $\varepsilon>0$ there is $T>0$ such that for $t \geq T$,
  \[
    \log \int_t^{\infty}\ee^{-\beta(1 + \varepsilon ) V(x)}  \dd x 
    \leq
    \log \int_t^{\infty}\ee^{-\beta(
      \lambda x + V(x))}  \dd x\le 
    \log \int_t^{\infty}\ee^{-\beta V(x)}  \dd x.
  \]
  Thus, we only need to show that
  \[
    \frac{1}{t^\gamma}\log \int_t^{\infty}\ee^{-\beta(1+\varepsilon)
      \frac{x^\gamma}{\gamma}} \dd x \xrightarrow[t \to \infty]{} -
    \frac{\beta(1+\varepsilon)}{\gamma}.
  \]
  for every $\varepsilon \geq 0$. This can be obtained by the change of
  variables $s = x/t$ and by using Laplace's method since the minimum of
  $\beta (1+\varepsilon) \frac{s^\gamma}{\gamma} $ for $s \in [1,\infty)$ is
  attained at $s=1$. For $V$ chosen as in \eqref{eq:SquaredZeroPotential} or
  \eqref{eq:LeftHardEdgePotential} we only need to use that $V(x) = 0$ for
  $x \geq 0$, and that
  \[
    \int_t^\infty\ee^{-\beta \lambda x} \dd x %
    = \frac{\ee^{-\beta \lambda t}}{\beta \lambda}.
  \]
\end{proof}

\section{Stochastic domination and conditioning}
\label{se:stodom}

Note that \emph{stochastic domination} is also known as \emph{stochastic
  monotonicity}. Throughout this appendix, we use ``density'' to refer to a
Radon\,--\,Nikodym derivative.

\begin{lemma}[Domination from non-decreasing density]
  \label{prop:IncreasingDensity}
  Let $\mu$ and $\nu$ be two probability measures for which there exists a
  non-decreasing measurable function $\rho: \mathbb R \to [0,\infty)$ such
  that $\dd \nu = \rho \,\dd \mu$. If $X \sim \mu$ and $Y \sim \nu$ then, for
  every $t \in \mathbb R$, $\mathbb P(Y \leq t) \leq \mathbb P(X \leq t)$.
\end{lemma}

\begin{proof}
  Since $\rho$ is non-decreasing there exists $a \in [-\infty,\infty]$ such
  that
  \[
    \rho(a) \leq 1 \text{ for }
    x \in (-\infty,a)
    \quad \text{ and }\quad 
    \rho(a) \geq 1 \text{ for }
    x \in (a,\infty).
  \]
  
  If $ t \leq a$ we have that
  $\displaystyle \mathbb P(Y \leq t) = \int_{(-\infty,t]} \rho \, \dd \mu%
  \leq\int_{(-\infty,t]} \dd \mu =\mathbb P(X \leq t)$.
  
  If $t > a$ we have that
  $\displaystyle\mathbb P(Y > t) = \int_{(t,\infty)}\rho \, \dd \mu %
  \geq\int_{(t,\infty)} \dd \mu =\mathbb P(X > t)$.

  Hence $\mathbb P(Y \leq t) %
  = 1-\mathbb P(Y > t) %
  \leq 1-\mathbb P(X > t) %
  = \mathbb P(X \leq t).$
\end{proof}

\begin{lemma}[Conditioning by steps]\label{lem:ConditionBySteps}
  Let $X_1,\dots,X_n$ be independent random variables such that
  $\mathbb P(X_n \leq \dots \leq X_1) > 0$. Fix $k \in \{1,\dots,n\}$ and
  consider a random variable $Y_k$ such that, by possibly enlarging the
  probability space,
  \[
    \mathrm{Law}(Y_k)=\mathrm{Law} (X_k\mid X_n \leq \dots \leq X_k)
  \]
  and $Y_k$ is independent of 
  $(X_1,\dots,X_{k-1})$.
  Then, 
  $\mathbb P(Y_k \leq X_{k-1} \leq \dots \leq X_1)>0$ 
  and
  \[
    \mathrm{Law}(X_k,\dots,X_1\mid X_n \leq \dots \leq X_1)
    =\mathrm{Law}(Y_k,X_{k-1},\dots,X_1\mid Y_k \leq X_{k-1} \leq \dots \leq X_1).
  \]
\end{lemma}

\begin{proof}
  Given a probability space $(\Omega, \mathcal{F}, \mathbb{P})$ and some
  $A\in\mathcal{F}$ such that $\mathbb P(A)>0$, let us use the notation
  $\mathbb P_A(C)=\mathbb P(C \cap A)/\mathbb P(A)$. For any
  $A, B\in\mathcal{F}$ such that $\mathbb P(A\cap B)>0$, we have that
  \begin{equation}\label{eq:ConditionOfCondition}
    \mathbb P_A(B) > 0 \quad \text{ and } \quad
    \left(\mathbb P_A\right)_B = 
    \mathbb P_{A\cap B}.
  \end{equation}
  In our setting, we consider $\Omega = \mathbb R^n$ with $\mathbb P$ given by
  the law of $(X_n,\dots,X_1)$. We have
  \[
    A = \left\{(x_n,\dots,x_1) \in \mathbb R^n:\, x_n \leq \dots \leq x_k
    \right\}
    \quad\text{and}\quad
    B = \left\{(x_n,\dots,x_1) \in \mathbb R^n: \, x_k \leq \dots \leq x_1
    \right\}.
  \]
  Let $Y_k,\dots,Y_n$ be random variables such that
  \[
    \mathrm {Law}(Y_n,\dots,Y_k)
    =\mathrm{Law}(X_n,\dots,X_k\mid X_n\leq \dots \leq X_k).
  \]
  Then, $\mathbb P_A$ is the law of $(Y_n,\dots,Y_k,X_{k-1},\dots,X_1)$ and the
  left-hand side of \eqref{eq:ConditionOfCondition} tells us that
  \[
    \mathbb P(Y_k \leq X_{k-1} \leq \dots \leq X_1) = \mathbb P_A(B)>0
  \]
  while the right-hand side of \eqref{eq:ConditionOfCondition} tells us that
  \[
    \mathrm{Law}(Y_n,\dots,Y_k,X_{k-1},\dots,X_1
    \mid  Y_k \leq X_{k-1} \leq \dots
    \leq X_1) = \mathrm{Law} (X_n,\dots,X_1
    \mid X_n \leq \dots \leq X_1).
  \]
  In particular, we have that
  \[
    \mathrm{Law}(Y_k,X_{k-1},\dots,X_1 \mid
    Y_k \leq X_{k-1} \leq \dots \leq X_1)
    = \mathrm{Law} (X_k,\dots,X_1 \mid
    X_n \leq \dots \leq X_1)
  \]
  which is the sought property.
\end{proof}

The non-decreasing density condition of Lemma \ref{prop:IncreasingDensity} is
preserved under order conditioning as the following proposition states.

\begin{proposition}[Preservation of non-decreasing densities under ordering]
  \label{prop:Hereditary}
  Let $X_1,\dots,X_n$, $Y_1,\dots,Y_n$ be independent real random variables
  such that $X_i \sim \mu_i$ and $Y_i \sim \nu_i$ where $\nu_i$ is absolutely
  continuous with respect to $\mu_i$ with a non-decreasing density
  $\rho_i=\dd \nu_i/\dd \mu_i$. If we define
  \[
    \widetilde \mu
    = \mathrm{Law} \left( X_1\mid
      X_n \leq \dots \leq X_1 \right)
    \quad\text{and}\quad
    \widetilde \nu
    = \mathrm{Law} \left( Y_1\mid
      Y_n \leq \dots \leq Y_1 \right)
  \]
  then $\widetilde\nu$ is absolutely continuous with respect to
  $\widetilde\mu$, with non-decreasing density
  $\widetilde\rho=\dd \widetilde\nu/\dd \widetilde\mu$.
\end{proposition}

\begin{proof}
  By Lemma \ref{lem:ConditionBySteps}, if we take random variables
  $\widetilde X_2$ and $\widetilde Y_2$ such that
  \[
    \mathrm{Law}(\widetilde X_2)
    = \mathrm{Law} \left( X_2\mid X_{n} \leq \dots \leq X_2 \right)
    \quad\text{and}\quad
    \mathrm{Law}(\widetilde Y_2)
    = \mathrm{Law} \left( Y_2\mid Y_n \leq \dots \leq Y_2\right)
  \]
  and such that $\widetilde X_2$ is independent of $X_1$ and $\widetilde Y_2$
  is independent of $Y_1$, we have
  \[
    \widetilde \mu = \mathrm{Law}( X_1\mid\widetilde X_2 \leq X_1)
    \quad\text{and}\quad
    \widetilde \nu = \mathrm{Law}( Y_1\mid\widetilde Y_2 \leq Y_1).
  \]
  So, by induction, it is enough to prove the lemma for $n=2$. In this case,
  $\widetilde \mu$ has a density with respect to $\mu_1$ proportional to
  $x\in \mathbb R \mapsto \mu_2(-\infty,x]$ and $\widetilde \nu$ has a density
  with respect to $\mu_1$ proportional to
  \[
    x\in \mathbb R
    \mapsto
    \rho_1(x)\, \nu_2(-\infty,x]
    =
    \rho_1(x)\, \int_{(-\infty,x]}
    \rho_2\, 
    \dd \mu_2.
  \]  
  In particular, $\widetilde \nu$ is absolutely continuous with respect to
  $\widetilde\mu$ and
  \[
    \dd \widetilde \nu 
    \propto 
    \frac{\rho_1(x)
      \displaystyle \int_{(-\infty,x]}
      \rho_2\, 
      \dd \mu_2} 
    {\mu_2(-\infty,x]}
    \,
    \dd \widetilde \mu.
  \]
  The proof is completed once we show that
  \[
    G(x) = \frac{\displaystyle \int_{(-\infty,x]} \rho_2 \, \dd \mu_2} {
      \displaystyle \mu_2(-\infty,x]}
  \]
  is non-decreasing. Note that $G$ is well-defined since if the denominator is
  zero then the numerator, which would be an integral over a measure zero set,
  is also zero.
  
  We must show that if $x \leq y$ then $G(x) \leq G(y).$ We know that
  $\rho_2(t) \leq \rho_2(s)$ for any $t \in (-\infty,x]$ and $s \in (x,y]$. By
  integrating over $(t,s) \in (-\infty,x] \times (x,y]$, with respect to
  $\mu_2\otimes\mu_2$, we obtain
  \[
    \Bigr(\int_{(-\infty,x]}\rho_2(t) \, \dd \mu_2(t)\Bigr)
    \mu_2(x,y]
    \leq
    \Bigr(\int_{(x,y]}\rho_2(s) \, \dd \mu_2(s)\Bigr)
    \mu_2(-\infty,x].
  \]
  Add $\bigr(\int_{(-\infty,x]} \rho_2 \, \dd \mu_2 \bigr)\mu_2(-\infty,x]$
  to both sides of the inequality to obtain
  \[
    \Bigr(\int_{(-\infty,x]}\rho_2 \, \dd \mu_2\Bigr)
    \mu_2(-\infty,y]
    \leq
    \mu_2(-\infty,x]
    \Bigr(\int_{(-\infty,y]}\rho_2 \, \dd \mu_2\Bigr),
  \]
  which, after dividing by
  $\mu_2(-\infty,x] \mu_2(-\infty,y]$, completes the proof.
\end{proof}

We next provide two applications of Proposition \ref{prop:Hereditary} that are
useful our context.

\begin{proposition}\label{prop:PushingToTheRight}
  Let $X_1,\dots,X_n$ be independent random variables with
  $\mathbb P(X_n \leq \dots \leq X_1) > 0$. Then, for every
  $1\leq k\leq m \leq n$ and $t \in \mathbb R$,
  \[
    \mathbb P \left(X_k \leq t \mid X_n \leq \dots \leq X_1\right) \, \leq \,
    \mathbb P \left(X_k \leq t \mid X_m \leq \dots \leq X_1 \right).
  \]
\end{proposition}

\begin{proof}
  By the same reasoning as in Lemma \ref{lem:ConditionBySteps} if we
  take a random variable $\widetilde X_k$ such that
  \[
    \mathrm{Law}(\widetilde X_k) = \mathrm{Law}(X_k\mid X_k \leq \dots \leq X_1)
  \]
  and independent of $X_1,\dots,X_{k-1}$, then
  \[
    \mathrm{Law} \left(X_k  \mid X_n \leq \dots \leq X_1 \right)
    =
    \mathrm{Law} \left(\widetilde X_k  \mid X_n \leq 
        \dots \leq X_{k+1} \leq
        \widetilde X_k\right)
  \]
  and
  \[
    \mathrm{Law} \left(X_k  \mid
      X_m \leq \dots \leq X_1\right)
    =
    \mathrm{Law} \left(\widetilde X_k  \mid X_m \leq 
      \dots \leq X_{k+1} \leq
      \widetilde X_k \right).
  \]
  Thus it suffices to consider the case $k=1$.
  It is enough to prove that
  \[
    \mathbb P \left(X_1 \leq t \bigm\vert
        X_{p+1} \leq \dots \leq X_1\right)
    \, \leq \,
    \mathbb P \left(X_1 \leq t \bigm\vert
        X_p \leq \dots \leq X_1\right)
  \]
  for every $p \in \{m,\dots,n-1\}$. For this, let
  $Y_p\sim\mathrm{Law}(X_p\mid X_{p+1} \leq X_p )$ be independent of
  $X_1,\dots,X_{p-1}$. Then, by Lemma \ref{lem:ConditionBySteps} we have
  \[
    \mathrm{Law} \left(X_1  \mid X_{p+1} \leq \dots \leq X_1 \right)
    =
    \mathrm{Law} \left(X_1\mid Y_{p} \leq X_{p-1} \leq \dots \leq X_1 \right).
  \]
  Since $Y_p$ has, with respect to the law of $X_p$, the non-decreasing density
  \[
    y \mapsto \frac{\mathbb P(X_{p+1} \leq y)} {\mathbb P(X_{p+1} \leq X_p)},
  \]
  by Proposition \ref{prop:Hereditary},
  $\mathrm{Law} \left(X_1 \mid Y_{p} \leq X_{p-1} \leq \dots \leq X_1 \right)$
  has a non-decreasing density with respect to the law
  $\mathrm{Law} \left(X_1 \mid X_{p} \leq X_{p-1} \leq \dots \leq X_1\right)$.
  It remains finally to use Lemma \ref{prop:IncreasingDensity}.
\end{proof}

\begin{lemma}\label{lem:HereditaryPotential}
  Let $t_0 \in [-\infty,\infty)$ and let $X_1,\dots,X_n$, $Y_1,\dots,Y_n$ be
  independent real random variables taking values on $[t_0,\infty)$ such that
  $X_i$ has density $\ee^{-g_i}$ and $Y_i$ has density $\ee^{-h_i}$ and such
  that $g_i-h_i$ is non-decreasing. Then, for every $t \in \mathbb R$,
  \[
    \mathbb P \left( Y_1\leq t\mid Y_n \leq \dots \leq Y_1 \right)
    \leq
    \mathbb P \left( X_1\leq t\mid X_n \leq \dots \leq X_1 \right).
  \]
\end{lemma}

\begin{proof}
  Since $Y_i$ has the non-decreasing density $\ee^{g_i-h_i}$ with
  respect to the law of $X_i$, we may use Proposition \ref{prop:Hereditary}
  and Lemma \ref{prop:IncreasingDensity} to complete the proof.
\end{proof}

\section{Conditional law of right-most particle}\label{sec:conditional}

The main results in this work concern point processes with infinitely many
particles. If one is concerned only with the right-most particle (or finitely
many particles) in a conditional setting, the proofs can be greatly
simplified. We illustrate this in the present appendix.

More specifically, Proposition \ref{pr:orderstats} is non-asymptotic, and gives the location of
the $k$ right-most particles of the gas, under conditioning, for an arbitrary
background of compact support. Up to the conditioning, this can be seen as some sort of
one-dimensional analog of a similar phenomenon for two-dimensional Coulomb
gases due to Kostlan \cite{MR1148410} and considered, for instance, in 
\cite{MR2552864,MR3215627,raphael-david,david-edge,MR4079754}. 
The proposition is reminiscent of a classical representation
theorem due to Alfred Rényi \cite{MR61792}, which states that if
$\{Z_i\}_{1\leq i\leq k}$ are independent and identically distributed
exponential random variables of unit mean, and if
$\{Z_{(j)}\}_{1\leq j\leq k}$ are the order statistics (max to min), then the
joint distribution of $\{Z_{(j)}\}_{1\leq j\leq k}$ is given by the identity
in distribution
\begin{equation}\label{eq:renyi}
\mathrm{Law}(Z_{(k)},\ldots,Z_{(1)})
=
\mathrm{Law}\Bigr(\frac{Z_1}{k},\ldots,\frac{Z_1}{k}+\cdots+\frac{Z_k}{1}\Bigr).
\end{equation}

We write $\mathrm{Card} \, S$ to denote the cardinality of a set $S$.

\begin{proposition}[Right-most particles outside the background]%
  \label{pr:orderstats}%
  Suppose that
  \begin{itemize}
  \item $\beta>0$ is fixed; 
  \item $\alpha>n-1$;
  \item $\rho$ is supported inside $(-\infty,0]$. 
  \end{itemize}
  Let us denote $X_{(i)}$ instead of $X^{(n)}_{(i)}$, and define
  $N_n=\mathrm{Card}\{1\leq i\leq n:X_i>0\}$.\\
  Then, for all $1\leq k\leq n$,
  \[ 
	\mathrm{Law}\left(X_{(k)}, \ldots, X_{(1)}\mid N_n=k\right)
	=\mathrm{Law}(Y_k,\ldots,Y_1\mid Y_k<\cdots<Y_1)
  \] 
  where ${(Y_i)}_{1\leq i\leq k}$ are independent exponential random
  variables with $\mathbb{E}Y_i=\frac{2}{\beta(\alpha-n-1+2i)}$.\\ %
  Alternatively, for all $1\leq k\leq n$,
  \[
	\mathrm{Law}\left(X_{(k)}, \ldots, X_{(1)}\mid N_n=k\right)
	=
	\mathrm{Law}\left(Z_1,Z_1+Z_2,\ldots,Z_1+\cdots+Z_k\right)
  \]
  where $(Z_i)_{1\leq i\leq k}$ are independent exponential random variables
  $\mathbb{E}Z_i=\frac{2}{\beta {i(\alpha-n+i)}}$.
\end{proposition}

\begin{proof}
  By \eqref{eq:CombinatorialIdentity} we know that
  \[
	\mathrm{Law}\left(X_{(n)}^{(n)},\dots, X_{(1)}^{(n)} \right) %
	=\mathrm{Law} \left(Y_n^{(n)},\dots,Y_1^{(n)} \Bigm\vert Y^{(n)}_n \leq
      \dots \leq Y^{(n)}_1 \right)
  \]
  where $Y_n^{(n)},\dots,Y_1^{(n)}$ are independent random variables and
  $Y_i^{(n)}$ has a density proportional to
  \[
	y \in \mathbb R %
	\mapsto %
	\exp \Bigr( %
	-\beta \Bigr[\frac{2k-n-1}{2}y - \alpha U_{\rho}(y) \Bigr] \Bigr).
  \]
  Then, we have that
  \begin{align*}
	\mathrm{Law}
	&\left(X_{(n)}^{(n)},\dots, X_{(1)}^{(n)} \Bigm\vert N_n = k \right)\\
	& \hspace{15mm}  =\mathrm{Law}
	\left(Y_n^{(n)},\dots,Y_1^{(n)}
   \Bigm\vert
   Y^{(n)}_{n}
   \leq \dots
   \leq
   Y^{(n)}_{k+1}
   \leq
   0
   <
   Y^{(n)}_k \leq \dots \leq Y^{(n)}_1 \right).
  \end{align*}
  Then, by the independence of $\left(Y_n^{(n)},\dots,Y_{k+1}^{(n)}\right)$ and
  $\left(Y_k^{(n)},\dots,Y_{1}^{(n)}\right)$, we have that
  \[
	\mathrm{Law} \left(X_{(k)}^{(n)},\dots, X_{(1)}^{(n)} \Bigm\vert N_n = k
	\right) =\mathrm{Law} \left(Y_k^{(n)},\dots,Y_1^{(n)} \Bigm\vert 0 <
      Y^{(n)}_k \leq \dots \leq Y^{(n)}_1 \right).
  \]
  By \eqref{eq:PotentialExplicitFormula} we can see that
  \[
	U_\rho(y)= -\frac{y}{2} + \frac{1}{2} \int_{\mathbb R} s \, \dd \rho(s).
  \]
  for $y > 0$ so that
  \[
	\mathrm{Law} \left(Y_k^{(n)},\dots,Y_1^{(n)} \Bigm\vert 0 < Y^{(n)}_k \leq
      \dots \leq Y^{(n)}_1 \right) = \mathrm{Law} \left(Y_k,\dots,Y_1 \mid Y_k
      \leq \dots \leq Y_1 \right)
  \]
  where $Y_i$ follows the law of $Y_i^{(n)}$ conditioned to be positive which
  has a density proportional to
  \[
	y \in (0,\infty) \mapsto e^{-\frac{\beta}{2}(\alpha-n-1+2i) y}.
  \]
  Then $\mathrm{Law}\bigr(X_{(k)}^{(n)},\dots, X_{(1)}^{(n)}\mid N_n=k\bigr)$ has a
  joint density proportional to
  \[
	(x_k,\dots,x_1) 
	\mapsto \exp \Bigr[-\frac{\beta}{2}
	\sum_{j=1}^k
	\left(\alpha - n +1 + 2j
	\right)x_j
	\Bigr]\mathbf{1}_{0 \leq x_k  \leq \dots \leq x_1}.
  \]
  We can perform the change of variables
  \[
	z_i = x_i - x_{i+1}\text{ if } 
	i \in \{1,\dots, k-1\}
	\quad \text{ and } \quad 
	z_k = x_k 
  \]
  or equivalently $x_j = \sum_{i=j}^k z_i$ for any $i \in \{1,\dots, k\}$ to
  obtain a density proportional to
  \[
	(z_k,\dots,z_1) \mapsto \prod_{i=1}^k
	\exp\left[-\frac{\beta}{2} (
      \alpha - n + i)
      i\,  z_i\right]\mathbf{1}_{z_i \geq 0}
  \]
  which implies the second assertion of the proposition.
\end{proof}

\subsection*{Acknowledgments.}{\small\ DGZ was supported by the French
  ANR-16-CE40-0024 SAMARA project. PJ was funded in part by the National
  Research Foundation of Korea (NRF) grants NRF-2017R1A2B2001952 and
  NRF-2019R1A5A1028324. PJ (respectively DC) thanks the hospitality of
  Université Paris-Dauphine -- PSL (respectively KAIST). Also, all authors
  thank the hospitality of CIRM at Luminy.}

\bibliographystyle{smfalpha}
\bibliography{jellium1d}

\end{document}